%% file: MirrorQuinticLGCY.tex
\documentclass[11pt]{amsart}

\usepackage{booktabs}
\usepackage{tabularx}
\usepackage{todonotes}

\usepackage[mathscr]{eucal}
\usepackage{mathpazo}
\usepackage{amsfonts, mathrsfs, amscd}
\usepackage[all]{xy}
\usepackage{url}
\usepackage{amssymb, amsmath, amsthm}
\usepackage{stackrel}

\usepackage{geometry}
\usepackage{array}
\usepackage{color}

\newtheorem{theorem}{Theorem}[section]

\newtheorem{lemma}[theorem]{Lemma}
\newtheorem{corollary}[theorem]{Corollary}
\newtheorem{conjecture}[theorem]{Conjecture}
\newtheorem{prop}[theorem]{Proposition}

\theoremstyle{definition}
\newtheorem{definition}[theorem]{Definition}

\theoremstyle{remark}
\newtheorem{remark}[theorem]{Remark}

\newtheorem*{convention}{Convention}

\DeclareMathOperator{\SL}{SL}   % SL: special linear group
\DeclareMathOperator{\Res}{Res}
\DeclareMathOperator{\ch}{ch}
\DeclareMathOperator{\mult}{mult}
\DeclareMathOperator{\age}{age}
\DeclareMathOperator{\id}{id}

\DeclareMathAlphabet{\mathpzc}{OT1}{pzc}{m}{it}

\makeatletter
\def\imod#1{\allowbreak\mkern10mu({\operator@font mod}\,\,#1)}
\makeatother

\newcommand{\ZZ}{\mathbb{Z}}
\newcommand{\LL}{\mathbb{L}}
\newcommand{\UU}{\mathbb{U}}
\newcommand{\CC}{\mathbb{C}}
\newcommand{\PP}{\mathbb{P}}
\newcommand{\ii}{\mathbb{1}}

\newcommand{\sH}{\mathscr{H}} 
\newcommand{\sL}{\mathscr{L}}
\newcommand{\sF}{\mathscr{F}}
\newcommand{\sV}{\mathscr{V}}
\newcommand{\sC}{\mathscr{C}}

\newcommand{\cO}{\mathcal{O}}   %% trivial bundle
\newcommand{\cM}{\overline{\mathcal{M}}}  %% Moduli of curves
\newcommand{\cW}{\mathcal{W}}
\newcommand{\cF}{\mathcal{F}}
\newcommand{\cD}{\mathcal{D}}
\newcommand{\cC}{\mathcal{C}}
\newcommand{\cL}{\mathcal{L}}
\newcommand{\cY}{\mathcal{Y}}
\newcommand{\cX}{\mathcal{X}}

\newcommand{\btau}{\boldsymbol{\tau}}

\newcommand{\jw}{\mathpzc{J}}
\newcommand{\bt}{\mathbf{t}}
\newcommand{\bc}{\mathbf{c}}
\newcommand{\bs}{\mathbf{s}}
\newcommand{\bp}{\mathbf{p}}
\newcommand{\bq}{\mathbf{q}}
\newcommand{\bv}[1]{\mathbf{#1}}

\newcommand{\on}{\operatorname}
\newcommand{\br}[1]{\left\langle#1\right\rangle}  % angle brackets
\newcommand{\set}[1]{\left\{#1\right\}}  % a set

\makeatletter
   \providecommand\@dotsep{5}
\makeatother

\begin{document}
\title{A Landau--Ginzburg/Calabi--Yau correspondence for the mirror quintic}
\author{Nathan Priddis}
\author{Mark Shoemaker}
\begin{abstract}
We prove a version of the Landau--Ginzburg/Calabi--Yau correspondence for the mirror quintic. In particular we calculate the genus--zero FJRW theory for the pair $(W, G)$ where $W$ is the Fermat quintic polynomial and $G = SL_W$.  We identify it with the Gromov--Witten theory of the mirror quintic three--fold via an explicit analytic continuation and symplectic transformation.
In the process we prove a mirror theorem for the corresponding Landau--Ginzburg model $(W,G)$.
\end{abstract}

\maketitle

\input{intro}

\input{FJRW}

\input{GW}

\input{Formalism}

\input{twistedFJRW}

\input{thecorrespondence}

\newpage

\bibliographystyle{plain}
\bibliography{references}

\end{document}

%% file: intro.tex
The Landau--Ginzburg/Calabi--Yau (LG/CY) correspondence was conjectured by physicists over twenty years ago based on mirror symmetry (\cite{VW, W}).  It describes a deep relationship between the geometry of Calabi--Yau complete intersections and the local structure of corresponding singularities.  Mathematically it was not until 2007, with the development of Fan--Jarvis--Ruan--Witten (FJRW) theory (\cite{FJR1}), that the conjecture could be made precise.  The LG/CY correspondence is now understood as relating the Gromov--Witten (GW) theory of a Calabi--Yau to the FJRW theory of the corresponding singularity (see Conjecture~\ref{conj}).  
Though interesting in its own right, there is also evidence to suggest that FJRW theory is easier to calculate than Gromov--Witten theory.  For example in \cite{Gu}, Gu\'er\'e calculates the genus-zero FJRW theory in a range of cases where the corresponding GW theory is currently unknown.
Thus the LG/CY correspondence provides a possible method of attack for determining the Gromov--Witten theory of many Calabi--Yau's.  

In genus zero the LG/CY correspondence has been proven in the case of a hypersurfaces in a Gorenstein weighted projective space in \cite{ChR, ChIR, KS} and extended to certain complete intersections in \cite{Cl}.  In this paper we prove a version of the genus zero correspondence for the mirror quintic, a Calabi--Yau hypersurface in an orbifold quotient of projective space.  This is the first case where the correspondence has been shown for a space which cannot be constructed as a complete intersection in weighted projective space.

\subsection{The quintic}  
We start by reviewing the LG/CY correspondence and its relation to mirror symmetry in the simpler case of the Fermat quintic.
We use the term \emph{Landau--Ginzburg model} to refer to a pair $(W, G)$ where $W$ is nondegenerate quasihomogeneous polynomial on $\CC^N$ and $G$ is a finite subgroup of $\text{Aut}(W)$.  We think of this data as defining a singularity $\set{W=0} \subset [\CC^N/G]$. In the case of the Fermat quintic, $W = x_1^5 + \cdots + x_5^5$, and we set $G=\langle \jw \rangle = \langle (\zeta_5, \ldots, \zeta_5)\rangle$ where $\zeta_5 = \exp(2 \pi i/5)$. Note that vanishing locus of $W$ also defines a Calabi--Yau variety in projective space, $M = \set{W = 0}$ in $\PP^4$.
  
On the Calabi--Yau side, one may define Gromov--Witten (GW) invariants 
\[
\langle \psi^{a_1} \alpha_{1}, \ldots,  \psi^{a_n} \alpha_{n} \rangle^M_{g, n, d}
\] 
as certain intersection numbers on the moduli space of stable maps from genus--$g$ curves into $M$.  On the Landau--Ginzberg side, Fan, Jarvis and Ruan (\cite{FJR1}) have defined  an analogous set of invariants, the FJRW invariants,
\[
\br{\psi^{a_1}\phi_{h_1},\dots, \psi^{a_n}\phi_{h_n}}^{(W,\langle \jw \rangle)}_{g,n}
\]
as intersection numbers on the moduli space of so--called $W$--curves (see Section~\ref{FJRW}).
The LG/CY correspondence for the quintic predicts a relationship between the Landau--Ginzburg model $(W, \langle \jw \rangle)$ and the hypersurface $M \subset \PP^4$ given by relating generating functions of the FJRW and GW invariants, respectively. 

In genus zero, the GW theory of $M$ is completely determined by Givental's $J$--function. More precisely, consider a basis for the \emph{state space} $H^*(M)$ of the GW theory of $M$. We can express a point of the state space in coordinates as $\bs = \sum_{i \in I} s^i\beta_i$.  The $J$--function is defined as a cohomology--valued  function,
\[
J^{M}(\bs, z) := z + \bs + \sum_{n \geq 0} \sum_{\substack{ a , d \geq 0 \\ h \in I}} \frac{Q^d}{n!z^{a+1}} \langle \bs, \ldots, \bs,  \psi^a\beta_h \rangle^{M}_{0, n+1,d} \beta^h.
\] 
In the case of the quintic hypersurface, the genus zero GW theory is completely determined by the restriction of the $J$--function to degree two. 

Mirror symmetry plays an important conceptual role in the correspondence. To state the mirror theorem for the quintic, we must consider the \emph{mirror family} to $M$. Let $W_\psi = x_1^5 + \cdots + x_5^5 - \psi x_1\cdots x_5$.  Define the group $\bar{G} \cong (\ZZ/5\ZZ)^3$ as the subgroup of the big torus of $\PP^4$ acting via generators $e_1, e_2, e_3$:
 \begin{align} \nonumber
  e_1[x_1, x_2, x_3, x_4, x_5] &= 
  [\zeta x_1, x_2, x_3, x_4, \zeta^{-1} x_5] \\ 
  e_2[x_1, x_2, x_3, x_4, x_5] &=  
  [ x_1, \zeta x_2, x_3, x_4, \zeta^{-1} x_5] \label{e:ees} \\ 
  e_3[x_1, x_2, x_3, x_4, x_5] &= 
  [ x_1, x_2, \zeta x_3, x_4, \zeta^{-1} x_5] . \nonumber
 \end{align}
Define the family $\cW_\psi$ of \emph{mirror quintics} as the family (of orbifolds) 
\[
\cW_\psi := \set{W = 0} \subset [\PP^4/\bar{G}].
\]  
The mirror theorem for the Fermat quintic, as formulated by Givental (\cite{G1}) states that after after a change of variables, the components of $J^{M}(\bs, z)$ restricted to $ s$ in $H^2(\cW)$ give a basis of solutions to the Picard--Fuchs equations for $\cW_\psi$ around the point $\psi = \infty$.

In the same way that the GW theory of $M$ corresponds to the family $\cW_\psi$ around $\psi = \infty$, it has recently been proven that the FJRW theory of $(W, \langle \jw \rangle)$ corresponds to the same family in a neighborhood of $\psi=0$. In this case the state space of the theory is defined in terms of Lefshetz thimbles of the singularity, and one may also define an FJRW $J$--function, $J^{(W, \langle \jw \rangle)}(\bt, z)$, in exact analogy to GW theory.  In \cite{ChR} it is proven that after a change of variables, the restriction of $J^{(W, \langle \jw \rangle)}(\bt, z)$ to  $t$ of degree two gives a basis of solutions to the Picard--Fuchs equations for $\cW_\psi$ around the point $\psi = 0$. As in the GW theory of M, the genus zero FJRW theory is completely determined by the restriction of the $J$--function to degree two.

The GW theory and FJRW theory correspond to Picard--Fuchs equations in neighborhoods of  $\psi = \infty$ and $\psi = 0$ respectively.  Thus we obtain as a corollary to the above mirror theorems that the genus zero GW theory of $M$ may be identified with the genus zero FJRW theory of $(W,\langle \jw \rangle)$ via analytic continuation in $\psi$ and a linear transformation $\UU$.  This is the LG/CY correspondence in genus zero.

Finally, Givental's symplectic formalism gives a possibility of determining the  higher genus LG/CY correspondence from the genus zero correspondence. Namely, it has been conjectured \cite[Conjecture 3.2.1]{ChR} that the quantization of $\UU$ should relate the higher genus invariants of the two respective theories.

\subsection{The mirror quintic}
Given the deep connection between the quintic $M$ and the mirror quintic $\cW := \cW_0$, it is natural to ask if we can prove a similar LG/CY correspondence for $\cW$. In this paper we prove that such a correspondence exists in this case as well.  Consider the one--parameter family of deformations of $M$ given by $M_\psi := \set{W_\psi = 0} \subset \PP^4$.  This family is mirror to $\cW$.

In analogy to the original mirror theorem, it was proven in \cite{LSh} that the orbifold GW theory of $\cW$ may be related to the Picard--Fuchs equations of $M_\psi$ around $\psi = \infty$. To be more precise, it was shown that after restricting to the degree two part of the \emph{untwisted} subspace of the state space, $H^2_{un}(\cW) \subset H_{CR}^*(\cW)$, the components of the mirror quintic $J$--function, $J^{\cW}(s,z)$, give solutions to the Picard--Fuchs equations of a holomorphic $(3,0)$--form on $M_\psi$ around $\psi = \infty$.  It was shown furthermore that the first derivatives 
\[
J_g^{\cW}(s, z) := z\frac{\partial J^{\cW}(\bs, z)}{\partial s^g} \Big|_{s \in H^2_{un}(\cW)}
\] 
give solutions to the Picard--Fuchs equations for the other (non--holomorphic) families of 3--forms on $M_\psi$ around $\psi = \infty$.  

As we will show, FJRW theory gives an analogous statement near the point $\psi= 0$. Consider the group
\begin{equation}\label{e:slw}
G := \langle \jw, e_1, e_2, e_3 \rangle \cong (\ZZ/5\ZZ)^4,
\end{equation}
where $e_1, e_2,$ and $e_3$ are as in \eqref{e:ees}. The genus zero FJRW theory of $(W, G)$ can be identified with solutions to the Picard--Fuchs equations of $M_\psi$ around $\psi = 0$ as follows:

\begin{theorem}[see Remark~\ref{r:mirror}] \label{t:intromirror}
After restricting to an appropriate one--parameter subset $\bt = t$, the FJRW $J$--function $J^{(W,G)}(t, z)$ satisfies, up to a change of variables, the Picard--Fuchs equations of a holomorphic $(3, 0)$--form on $M_\psi$ around $\psi = 0$.  Furthermore, the first derivatives $J_h^{(W,G)}(t, z)$ give solutions to the Picard--Fuchs equations of the other (non--holomorphic) families of 3--forms on $M_\psi$ around $\psi = 0$. 
\end{theorem}

Solutions to the Picard--Fuchs equations for the family $M_\psi$ around $\psi = 0$ and $\psi = \infty$ are related by analytic continuation.  So as a corollary to the above theorem, we obtain an explicit relationship between the respective $J$--functions $J^\cW(s, z)$ and $J^{(W, G)}(t, z)$ and their derivatives.  
\begin{theorem}[= Theorem~\ref{t:U}]
There exists a linear symplectic transformation $\UU$ which (after a mirror transformation) identifies derivatives of the analytic continuation of $J^{\cW}(\bs, z)$ with those of $J^{(W, G)}(\bt, z)$.  Such a $\UU$ is unique up to a factor and choice of analytic continuation.
\end{theorem}
The consideration of not just $J^{\cW}(\bs, z)$ and $J^{(W, G)}(\bt, z)$ but also their derivatives is necessary to uniquely determine $\UU$.

Givental's symplectic formalism allows one to rephrase the above theorem in a more useful form.  In this setting, one can view the genus zero generating functions of GW theory and FJRW theory as generating \emph{Lagrangian cones} 
$\sL^\cW$ and $\sL^{(W,G)}$ in appropriate symplectic vector spaces.  These Lagrangian subspaces completely determine the respective genus zero theories.  The above theorem then identifies a certain subset of $\sL^\cW$, the \emph{small slice} (see Definition~\ref{d:small}) of $\sL^{\cW}$, with the corresponding slice of $\sL^{(W,G)}$.

\begin{theorem}[= Theorem~\ref{t:main}]
The symplectic transformation $\UU$ identifies the analytic continuation of the small slice of $\sL^{\cW}$ with the small slice of $\sL^{(W,G)}$.
\end{theorem}
As in the case of the quintic, it is conjectured that the quantization of $\UU$ identifies the (analytic continuation of the) higher genus GW theory of $\cW$ with the FJRW theory of $(W,G)$.

\subsection{Acknowledgements}
The authors wish to express their gratitude to Prof. Y. Ruan for his help and guidance over the years.  They would also like to thank E. Clader, D. Ross, and Y. Shen for helpful discussions on Givental's theory and A Chiodo for giving useful comments on our paper. .  

M.S. is grateful to Prof. Y. P. Lee for the collaboration which made this work possible, as well as many helpful discussions about the current project.  

This material is based upon work supported by the National Science Foundation under NSF RTG grant 1045119 and NSF FRG grant 1159265.

%% file: FJRW.tex
\section{LG--Theory}\label{FJRW}

For the Mirror Quintic, the LG model is described by FJRW--theory. Here we will give a brief review of the definitions and facts we will need to describe the LG/CY correspondence (see \cite{FJR1} or \cite{ChR}) . The mirror theorem for the LG model will be given in Section~\ref{twisted}. 

%%%% STATE SPACE %%%%
\subsection{State Space}

A polynomial $W\in \CC[x_1,\dots,x_N]$ is \emph{quasihomogeneous} of degree $d$ with integer weights $w_1, \dots, w_N$ if for every $\lambda \in \CC$,
\[
W(\lambda^{w_1}x_1,\dots, \lambda^{w_N}x_N) = \lambda^dW(x_1, \dots, x_N).
\]
By rescaling the numbers $w_1, \dots, w_N$ and $d$, we can require that $\gcd(w_1, \dots, w_N)=1$. For each $1\leq i\leq N$, let $q_k=\tfrac{w_k}{d}$. The central charge of $W$ is defined to be
\begin{equation}
\hat c:=\sum_{k=1}^N(1-2q_k).
\end{equation}

A polynomial is \emph{nondegenerate} if 
\begin{itemize}
\item[(i)] the weights $q_k$ are uniquely determined by $W$, and
\item[(ii)] the hypersurface defined by $W$ is non--singular in projective space. 
\end{itemize}

%%---Group: def,  order, generators, %%

The maximal group of diagonal symmetries is defined as
\[
G_{max}:=\set{(\alpha_1,\dots,\alpha_N)\subseteq (\CC^*)^N \,|\,W(\alpha_1x_1,\dots,\alpha_Nx_N)=W(x_1,\dots,x_N)}
\]
Note that $G_{max}$ always contains the \emph{exponential grading element} $\jw=(e^{2\pi iq_1},\dots,e^{2\pi iq_N})$. If $W$ is nondegenerate, $G_{max}$ will be finite.  Define the \emph{exponent of $W$}, denoted $\bar{d}$, as the order of the largest cyclic subgroup of $G_{max}$. In this paper, we will assume for simplicity that $\bar{d}$ is equal to the degree $d$ of $W$.  This does not hold in general, but will be true in the case of interest to us.

A group $G\subset G_{max}$ is \emph{admissible} if there is a Laurent polynomial $Z$, quasihomogeneous with the same weights as $W$, having no monomials in common with W, such that the maximal group of diagonal symmetries of $W+Z$ is equal to $G$. Every admissible group $G$ has the property that $\jw \in G$. Let $\SL_W=\set{(\alpha_1,\dots, \alpha_N)\in G_{max}| \prod \alpha_i=1}$. 
If $W$ satisfies the \emph{Calabi--Yau} condition $\sum_{k=1}^N w_k = d$, then $Z=x_1x_2\dots x_N$ will be quasihomogeneous, thus $\SL_W$ will be admissible. 

Let $G$ be an admissible group. For $h\in G$, let $\CC^N_h$ denote the fixed locus of $\CC^N$ with respect to $h$.
Let
$N_h$ be the complex dimension of the fixed locus of $h$. Define
\[
\sH_{h}:=H_{N_h}(\CC_h^{N}, W_h^{+\infty}; \CC)^{G},
\]
that is, $G$--invariant elements of the the middle dimensional relative cohomology of $\CC_h^{N}$. Here $W_h^{+\infty}:=(\Re W_h)^{-1}(\rho, \infty)$, for $\rho>> 0$. The state space is the direct sum of the ``sectors'' $\sH_h$, i.e. 
\[
\sH_{W,G} := \bigoplus_{h\in G}\sH_h.
\]

$\sH_{W,G}$ is $\mathbb{Q}$--graded by the $W$--degree. To define this grading, first note that each element $h\in G$ can be uniquely expressed as
\[
h=(e^{2\pi i\varTheta_1(h)},\dots,e^{2\pi i\varTheta_N(h)})
\]
with $0\leq \varTheta_k(h) < 1$. 
The \emph{degree--shifting number} is 
\[
\iota(h):=\sum_{k=1}^N(\varTheta_k(h)-q_k).
\]
For $\alpha_h \in \mathscr{H}_{h}$, the (real) $W$--degree of $\alpha_h$ is defined by
\begin{equation}\label{e:degw}
\deg_W(\alpha_h):=N_h+2\iota(h).
\end{equation}

\begin{remark}\label{r:identity element}
Although we will not need it in this paper, one can define a product structure on $\sH_{W,G}$, which then becomes a graded algebra.
Let $\phi_\jw$ be the fundamental class in $\sH_\jw$, and note that $\deg_W(\phi_\jw) = 0$.  In fact $\phi_\jw$ is the identity element in $\sH_{W,G}$.  This partially explains the prominence of the element $\jw$ in the above discussion.
\end{remark}

There is also a non--degenerate pairing
\[
\br{-,-}:\sH_h\times\sH_{h^{-1}}\to \CC,
\]
which induces a symmetric non--degenerate pairing, 
\[
\br{- ,- }:\sH_{W,G}\times \sH_{W,G}\to \CC.
\]

%%%%%% MODULI SPACE OF W--CURVES %%%%%%%%%

\subsection{Moduli of W--curves}

Recall that an \emph{$n$--pointed orbifold curve} is a stack of Deligne--Mumford type with at worst nodal singularities with orbifold structure only at the marked points and the nodes. We require the nodes to be \emph{balanced}, in the sense that the action of the stabilizer group be given by 
\[
(x,y)\mapsto (e^{2 \pi i/ k} x,e^{-2 \pi i/k}  y).
\]

Given such a curve, $\cC$, let $\omega$ be its dualizing sheaf. The \emph{log--canonical bundle} is 
\[
\omega_{\log}:=\omega(p_1+ \dots + p_n)
\]

Following \cite{ChR}, we will consider \emph{$d$--stable} curves. A $d$--stable curve is a proper connected orbifold curve $\cC$ of genus $g$ with $n$ distinct smooth markings $p_1,\dots,p_n$ such that
\begin{enumerate}
\item[(i)] the  $n$--pointed underlying coarse curve is stable, and
\item[(ii)] all the stabilizers at nodes and markings have order $d$. 
\end{enumerate}

There is a moduli stack, $\cM_{g,n,d}$ parametrizing such curves. It is proper, smooth and has dimension $3g-3+n$. (As noted in \cite{ChR}, it differs from the moduli space of curves only because of the stabilizers over the normal crossings.) 

Write $W$ as a sum of monomials $W=W_1+\dots+W_s$, with $\displaystyle W_l=c_l\prod_{k=1}^N x_k^{a_{lk}}$. Given line bundles $\cL_1, \ldots , \cL_N$ on the $d$--stable curve $\cC$, we define the line bundle \[W_l(\cL_1,\dots,\cL_N) := \bigotimes_{k=1}^N\cL_k^{\otimes a_{lk}}.\] 

\begin{definition}A \emph{$W$--structure} is the data $(\cC, p_1,\dots, p_n, \cL_1,\dots,\cL_N,\varphi_1,\dots \varphi_N)$, where $\cC$ is an $n$--pointed $d$--stable curve, the $\cL_k$ are line bundles on $\cC$ satisfying 
\begin{equation}\label{e:poly}W_l(\cL_1,\dots,\cL_N)\cong \omega_{\log},
\end{equation} and for each $k$, $\varphi_k:\cL_k^{\otimes d}\to \omega_{\log}^{w_k}$ is an isomorphism of line bundles. 
\end{definition}
There exists a moduli stack of $W$--structures, denoted by $W_{g,n}$.

\begin{prop}[\cite{ChR}]
The stack $W_{g,n}$ is nonempty if and only if $n>0$ or $2g-2$ is a positive multiple of $d$. It is a proper, smooth Deligne--Mumford stack of dimension $3g-3+n$. It is etale over $\cM_{g,n,d}$ of degree $|G_{max}|^{2g-1+n}/d^N$. 
\end{prop}

The moduli space can be decomposed into connected components, which we now describe. Because $\cL_k$ is a $d$th root of a line bundle pulled back from the coarse underlying curve, the generator of the isotropy group at $p_i$ acts on $\cL_k$ by multiplication by $e^{2 \pi i m_k^i/d}$ for some $0 \leq m_k^i < d$.  The integer $m_k^i$ is the \emph{multiplicity} of $\cL_k$ at $p_i$, and will usually be denoted $\mult_{p_i}(\cL_k)$.
Equation~\eqref{e:poly} ensures that $(e^{2\pi i m^i_1/d}, \dots,e^{2\pi im^i_N/d})\in G_{max}$. Furthermore, when we push forward the line bundle $\cL_k$ to the coarse curve, we find it has degree
\begin{equation}\label{e:degree}
q_k(2g-2+n)-\sum_{i=1}^n \mult_{p_i}(\cL_k)/d,
\end{equation}
which must therefore be an integer.

Let $\bv{h}=(h_1,\dots,h_n)$ denote an $n$--tuple of group elements, $h_i \in G_{max}$. Define $W_{g,n}(\bv{h})$ to be the stack of $n$--pointed, genus $g$ $W$--curves for which $\mult_{p_i}(\cL_k)/d = \varTheta_k(h_i)$. The following proposition describes a decomposition of $W_{g,n}$ in terms of multiplicities:
\begin{prop}[\cite{ChR,FJR1}]\label{p:1.3}
The stack $W_{g,n}$ can be expressed as the disjoint union
\[
W_{g,n}=\coprod W_{g,n}(\bv{h})
\]
with each $W_{g,n}(\bv{h})$ an open and closed substack of $W_{g,n}$. Furthermore, $W_{g,n}(\bv{h})$ is non--empty if and only if
\begin{align*}
h_i &\in G_{max}, \: i=1,\dots, n\\
q_k(2g-2+n)-\sum_{i=1}^n \varTheta_k(h_i) &\in \ZZ, \: \: \:\: \: \: \: k=1,\dots, N.
\end{align*}
\end{prop}

Suppose $G\subset G_{max}$ is an admissible group, so $G$ is the maximal group of diagonal symmetries of $W+Z$ for some choice of quasihomogeneous Laurent polynomial $Z$. We define $W_{g,n,G}$ to be the stack of $(W+Z)$--curves with genus $g$ and $n$ marked points. This definition does not depend on the particular choice of $Z$ (see \cite{FJR1}). 

\begin{prop}[\cite{ChR, FJR1}]
$W_{g,n,G}$ is a proper substack of $W_{g,n}$. 
\end{prop}
We denote the universal curve by $\pi:\sC\to W_{g,n,G}(\bv{h})$, and the universal $W$--structure by $(\LL_1,\dots, \LL_N)$. 

For each substack $W_{g,n}(\bv{h})$, one may define a virtual cycle $[W_{g,n}(\bv{h})]^{vir}$ of degree 
\[
2\left((\hat c -3)(1-g)+ n -\sum_{i=1}^n \iota(h_i)\right).
\]
The virtual cycle $[W_{g,n,G}(\bv{h})]^{vir}$ is defined as
\[
[W_{g,n,G}(\bv{h})]^{vir}:=\frac{|G_{max}|}{|G|}i^*[W_{g,n}(\bv{h})]^{vir},
\]
with $i:W_{g,n,G}(\bv{h})\hookrightarrow W_{g,n}(\bv{h})$ the inclusion map.

The stacks $W_{g,n}$ are also equipped with $\psi$--classes. We define $\psi_i$ as the first Chern class of the bundle whose fiber over a point is the cotangent line to the corresponding coarse curve at the $i$th marked point.

%%%%%%%%%%%%%  FJRW INVARIANTS  %%%%%%%%%%%%%%%%%%

\subsection{FJRW Invariants}\label{s:FJRWinv}

FJRW invariants can be defined for any pair $(W,G)$ where $W$ is a nondegenerate quasihomogeneous polynomial and $G$ is an admissible group. However, the most general definition is somewhat complicated, and unnecessary for our purposes here. To simplify the exposition, we will specialize to the case of interest to us, namely $W=x_1^5+\dots+x_5^5$ and $G=\SL_W$. 

$W$ is degree five with weights are $w_k  = 1$ for $1 \leq k \leq 5$. In this case $\jw=(e^{2\pi i/5}, \dots, e^{2\pi i/5})$, and $\hat c=3$. The group $G_{max}$ is isomorphic to $(\ZZ_5)^5$, and the subgroup $G = \SL_W$ is defined in \eqref{e:slw}.
By a slight abuse of notation, we will often represent a group element $h=(e^{2\pi i\varTheta_1(h)},\dots,e^{2\pi i\varTheta_N(h)})$ by
\[
h = (\varTheta_1(h), \ldots , \varTheta_5(h)).
\]
With this convention, we can write 
\[
 G  = \set{ \left(\tfrac{r_1}{5}, \ldots , \tfrac{r_5}{5}\right) \, | \, \sum_{k=1}^5 r_k \equiv 0 \, \imod{5}, 
  0 \leq r_k \leq 4 }.
\]

In computing the state space, we find that the only non--zero sectors are the identity sector $\sH_e$, and those with $N_h=0$. If $N_h = 0$ we call $\sH_h$ a ``narrow'' sector. Let $\hat{S}=\set{h\in G| N_h=0}$ denote the index set for the narrow sectors. As each narrow sector is fixed by $G$, the state space can be decomposed as 
\[
\sH_{W,G}=\sH_e\oplus \bigoplus_{h\in \hat{S}}\sH_{h}.
\]
with $\sH_{h}\cong \CC$. The elements of $\sH_e$ have degree three. The elements of each of the narrow sectors have even W--degree. In what follows we will focus on the subspace of narrow sectors, 
\[\sH_{W,G}^{nar} := \bigoplus_{h\in \hat{S}}\sH_{h}.\]

There is an obvious choice of basis $\set{\phi_h}_{h\in \hat{S}}$, where $\phi_h$ is the fundamental class in  $\sH_{h}$. Let $\set{\phi^h}$ denote the dual basis with respect to the pairing, i.e. $\phi^h=\phi_{h^{-1}}$.

The moduli space may now be described as
\[
W_{g,n,G}=\set{(\cC, p_1,\dots, p_n, \cL_1,\dots,\cL_5,\varphi_1,\dots,\varphi_5)|\varphi_k:\cL_k^{\otimes 5} \stackrel{\sim}{\to} \omega_{\log}, \; \otimes_{k=1}^5\cL_k\cong \omega_{\log}}.
\]

For each $\bv{h}=(h_1,\dots, h_n)\in (G)^{n}$, the virtual cycle $[W_{g,n}(\bv{h})]^{vir}$ has degree 
\[
2\left((\hat c -3)(1-g)+ n -\sum_{k=1}^n \iota(h_k)\right)=2n-2\sum_{k=1}^n \iota(h_k).
\] 
Let $\bv{h}=(h_1,\dots,h_n)$. Define the FJRW invariant
\[
\br{\psi^{a_1}\phi_{h_1},\dots, \psi^{a_n}\phi_{h_n}}^{(W,G)}_{g,n}:=\frac{1}{625^{g-1}}\int_{[W_{g,n,G}(\bv{h})]^{vir}}\prod_{i=1}^n \psi_i^{a_i}. 
\]
Extending linearly, we obtain invariants defined for any insertions in $\sH_{W,G}$. 

The perfect obstruction theory used to define the virtual class is given by $-R\pi_*(\bigoplus_{k=1}^5\LL_k)^\vee$.  In genus zero, the situation simplifies greatly:

\begin{prop}\label{p:concave}
The genus zero FJRW theory for the mirror quintic is concave, and
\[
-R\pi_*\Big(\bigoplus_{k=1}^5\LL_k\Big)^\vee = R^1\pi_*\Big(\bigoplus_{k=1}^5\LL_k\Big)^\vee.
\]
\end{prop}

\begin{proof}
We will show that over any geometric point $(\cC,p_1,\dots,p_n,\cL_1,\dots,\cL_5,\varphi_1,\dots,\varphi_5)$ in the moduli space,  $\bigoplus_{k=1}^5 H^0(\cC,\cL_k) = 0$.  This then implies the result.
Let $f:\cC \to C$ denote the map from the stack $\cC$ to the coarse underlying curve $C$, and let $|\cL_k|$ denote the push forward $f_*\cL_k$.  Then  $H^0(\cC,\cL_k) \cong  H^0(C,|\cL_k|)$, thus it suffices to show that the line bundle $|\cL_k|$ has no global sections. 

Let $\Gamma$ be the dual graph to $C$ (see \cite{GrP}). Recall that each vertex $v$ of $\Gamma$ corresponds to a rational curve component $C_v$. Let $P_v$ denote the set of special points (marks and nodes) on $C_v$ and $k_v$ the number of such points. For $\tau\in P_v$, let $\mult_{\tau}(\cL_k)$ be the multiplicity of $\cL_k$ at the point $\tau$.    
As in equation~\eqref{e:degree}, the degree of the push forward $|\cL_k|_{C_v}$ can be expressed in terms of the multiplicity at each special point: 
\begin{align*}
\deg(|\cL_k|_{C_v})&=\tfrac 15(k_v-2)-\tfrac 15\sum_{\tau \in P}\mult_{\tau}(\cL_k) \\ &= -\tfrac 25+\tfrac 15\sum_{\tau\in P} \left(1-\mult_{\tau}(\cL_k)\right) 
\end{align*}

Since we have restricted our consideration to narrow sectors,  $\mult_{\tau}(\cL_k)> 0$ whenever $\tau$ is not a node. If $C$ is irreducible, we see that $\deg(|\cL_k|)$ is negative and $H^0(C,|\cL_k|) = 0$.  If $C$ is reducible, each component of $C_v$ has at least one node and we obtain
the following inequality:
\begin{equation}\label{e:deg}
\deg(|\cL_k|_{C_v}) \leq \tfrac 15(\#\text{nodes}(C_v)-2)<\#\text{nodes}(C_v)-1.
\end{equation}

Since we are in genus 0, $\Gamma$ is a tree. Choose one of the 1--valent vertices, $v$.  There is only one node on the corresponding rational component $C_v$.  By equation~\eqref{e:deg}, $\deg(|\cL_k|_{C_v})<0$ so any section of $|\cL_k|$ must vanish on $C_v$. 
Choosing a vertex attached to $t+1$ edges, \eqref{e:deg} yields $\deg(|\cL_k|_{C_v})<t$.  Therefore if a section of $|\cL_k|_{C_v}$ vanishes at $t$ of the nodes, we see by degree considerations that it must be identically zero on $C_v$.

By starting at the outer vertices of $\Gamma$ and working in, the above two facts allow one to show that a section of
$|\cL_k|$ must vanish on every component of $C$.
\end{proof}

On each $W$--curve in $W_{g, n, G}$ we have $\otimes_{k=1}^5 \cL_k \cong \omega_{\log}$.  This implies that $\cL_5$ is determined by $\cL_1$,$\cL_2$, $\cL_3$, $\cL_4$. We will use this fact to facilitate computation.

Let $(A_4)_{g,n}$ denote the moduli space of genus $g$, $n$--marked $A_4$--curves corresponding to the polynomial $A_4=x^5$. Such $W$--structures are often referred to as 5--spin curves.  Let $(A_4^4)_{g, n}$ denote the fiber product 
\[(A_4^4)_{g, n}:=
(A_4)_{g,n}\times_{\cM_{g,n,5}}(A_4)_{g,n}\times_{\cM_{g,n,5}}(A_4)_{g,n}\times_{\cM_{g,n,5}}(A_4)_{g,n}
\]

\begin{prop}There is a surjective map 
\[
s:(A_4^4)_{g, n}\to W_{g,n,G}
\]
which is a bijection at the level of a point. 
\end{prop}

\begin{proof}
The map is 
\[
(\cL_1,\dots,\cL_4,\phi_1,\dots,\phi_4)\to \Big(\cL_1,\dots,\cL_4,\Big(\big(\bigotimes_{k=1}^4 \cL_k^\vee\big)\otimes \omega_{log}\Big),\varphi_1,\dots,\varphi_4, \varphi_1^\vee\otimes\dots\otimes\varphi_4^\vee\otimes \id \Big). 
\]
Notice that the image of this map satisfies $\bigotimes_{k=1}^5\cL_k \cong \omega_{\log}$, and the fifth line bundle in the image is a fifth root of $\omega_{\log}$. Furthermore, every point in $W_{g,n,G}$ is of this form.  It is clear that this map is bijective at the level of a point.  This implies the proposition.
\end{proof}

Using the previous two propositions, we can give a more useful description of the genus zero correlators. Given $\bv{h}=(h_1,\dots,h_n)$, let us denote 
\[
A_4^4(\bv{h})_{g,n}:=(A_4)_{g,n}(\varTheta_1(h_1), \dots \varTheta_1(h_n))\times_{\cM_{g,n,d}}\dots\times_{\cM_{g,n,d}}(A_4)_{g,n}(\varTheta_4(h_1),\dots,\varTheta_4(h_n)).
\]
Each factor of $(A_4)_{g,n}$ is equipped with a universal $A_4$--structure. Abusing notation, we denote the universal line bundle over the $i$th factor of $(A_4^4)_{g, n}$ also by $\LL_i$.
By the universal properties of the $W$--structure on $W_{g,n}$, we have $s^*\LL_k\cong \LL_k$ for $1 \leq k \leq 4$. Define $\LL_5$ on $(A_4^4)_{g,n}$ as the pullback $s^*\LL_5$.

In \cite{FJR1} the authors show that $[W_{0,n,G}]^{vir}$ is Poincar\'{e} dual to $5c_{top}\Big(R^1\pi_*\big(\bigoplus_{i=1}^5 \LL_i\big)^\vee\Big)$ as a consequence of concavity. By the projection formula, we can pull the correlators back to $(A_4^4)_{0,n}$. The map $s$ has degree 5, so we get the following expression for the genus 0 correlators:
\[
\br{\psi^{a_1}\phi_{h_1},\dots, \psi^{a_n}\phi_{h_n}}^{(W,G)}_{0,n}=625\int_{A_4^4(\bv{h})_{0,n}} \prod_{i=1}^n \psi_i^{a_i}\cup c_{top}\Big(R^1\pi_*\big(\bigoplus_{i=1}^5 \LL_i\big)^\vee\Big)
\]

%% file: GW.tex
\section{Gromov--Witten theory of the mirror quintic}
Here we introduce the mirror quintic and describe its cohomology.

Recall the pair $(W,G)$ from Section~\ref{s:FJRWinv}.  Let $\bar{G}$ denote the quotient $G/\langle \jw \rangle$.  Let $\cY$ denote the global quotient orbifold
\[ \cY = [\PP^4/\bar{G}]\]
where the $\bar{G}$--action on $\PP^4$ comes from coordinate--wise multiplication.  
The mirror quintic $\cW$ is defined as the hypersurface
\[ 
\cW :=  \{ W = 0 \} \subset \cY.
\]

The correct cohomology theory for orbifold Gromov--Witten theory is \emph{Chen--Ruan orbifold cohomology}, defined via the \emph{inertia orbifold} (see \cite{ChenR1}).
If $\cX = [V/H]$ is a global quotient of a nonsingular
variety $V$ by a finite group $H$, the inertia orbifold $I\cX$ takes a particularly simple form.
Let $S_H$ denote the set of conjugacy classes $(h)$ in $H$,
then
\[
 I [V/H] = \coprod_{(h) \in S_H} [ V^h/C(h) ].
\]
As a vector space, the Chen--Ruan cohomology groups $H^*_{CR}(\cX)$ 
of an orbifold $\cX$ are the cohomology groups of 
its inertia orbifold:
\[
  H_{CR}^*(\cX) := H^*(I\cX).
\]

We will now describe the Chen--Ruan cohomology of the mirror quintic $\cW$.  For more detail, refer to \cite{LSh}.
For an element $g \in G$, denote by $[g]$ the corresponding element in $\bar{G}$ and $I(g) :=  \set{ k \in \{1, 2, 3, 4, 5\} \, | \, \varTheta_k(g) =0 }$. The order of this set is $N_g$ as defined in Section~\ref{FJRW}.

Fix an element $\bar{g} \in \bar{G}$. 
Given $g \in G$ such that $[g] =\bar{g}$, the set 
\[
  \PP^4_g := \left\{x_j = 0\right\}_{j \notin I(g)} \subset \PP^4
\] 
is a component of the fixed locus $(\PP^4)^{\bar{g}}$.  
From this we see that each element $g \in G$ such that 
$[g] =\bar{g}$ corresponds to a connected component $\cY_{g}$ of $I\cY$ 
associated with $\PP^4_g \subset (\PP^4)^{\bar{g}}$. 
Note that if $g$ has no coordinates equal to zero then $\PP^4_g$ 
is empty, and so is $\cY_g$.
This gives us a convenient way of indexing components of $I\cY$.  

Let
\[
 \cY_g = \{ (x, [g]) \in I\cY \, | \, x \in [\PP^4_g / \bar{G} ]\},
\]
and let $S$ denote the set of
all $g$ such that $\varTheta_k(g)$ is equal to $0$ for at least one $k$. Then
\begin{equation*}  
  I\cY = \coprod_{g \in S} \cY_{g}\:,
\end{equation*} 
with each $\cY_g$  a connected component. 

The inertia orbifold of the mirror quintic $\cW$ can be described in terms of $I\cY$.
The mirror quintic $\cW$ intersects nontrivially with $\cY_g$ exactly when $N_g \geq 2$. (that is, $\dim \cY_g \geq 1$.) 
Let 
\[ 
 \tilde{S} := \set{g \in G \big|\,  N_g \geq 2}.  
\]
Then 
\[
  I\cW = \coprod_{g \in \tilde{S}} \cW_g \,, \: \text{where }
   \cW_g = \cW \cap \cY_g.  
\] 
All nontrivial intersections are transverse, so 
\[
 \dim(\cW_g) = \dim(\cY_g) - 1 = N_g - 2.
\]
For $g \in \tilde{S}$, the \emph{age} of $g$ is defined as 
\[
\age(g) := \sum_{k = 1}^5 \varTheta_k(g).
\] 
The Chen--Ruan cohomology of $\cW$ is defined, as a graded vector space, by
\[ 
H^*_{CR}(\cW) :=  \bigoplus_{g \in \tilde{S}} H^{* - 2\on{age}(g)}(\cW_g).
\]
As in FJRW theory, we will only be interested in the subring of $H^*_{CR}(\cW)$
consisting of classes of even (real) degree.  We will denote this 
ring as $H^{even}_{CR}(\cW)$.

Let $\ii_g$ denote the constant function with value one on $\cW_g$.  Let $\overline{H}$ denote the class in $H^*(\cY)$ which pulls back to the hyperplane class in $\PP^4$ and $H$ the induced class on $\cW$.

A convenient basis for $H_{CR}^{even}(\cW)$ is 
\begin{equation*} 
  \bigcup_{g \in \tilde{S}}   \{\ii_g, \ii_g H,\ldots , 
  \ii_g {H}^{\dim(\cW_g)}\}.
\end{equation*}

Let $s$ represent the dual coordinate to $H \in H^*_{CR}(\cW)$.  We denote by $H^2(\cW)$ the subspace $sH$ of  classes in $H^2_{CR}(\cW)$ supported on the \emph{untwisted component} $\cW \subset I\cW$.

%% file: Formalism.tex
\section{Givental formalism and the LG/CY conjecture for the mirror quintic}\label{formal}

Similar to Section~\ref{FJRW}, given a smooth orbifold $\cX$, one may define Gromov--Witten invariants $\langle \psi^{a_1} \alpha_{1}, \ldots,  \psi^{a_n} \alpha_{n} \rangle^\cX_{g, n, d}$, where $d$ is the degree of the map from the source curve into $\cX$ and $\beta_i\in H_{CR}^{even}(\cX)$ (see e.g. \cite{AGrA} or \cite{ChenR2}).
Summing over the degree, we write  
\[
\langle \psi^{a_1} \alpha_{1}, \ldots,  \psi^{a_n} \alpha_{n} \rangle^\cX_{g, n}:=\sum_{d} Q^d \langle \psi^{a_1} \alpha_{1}, \ldots,  \psi^{a_n} \alpha_{n} \rangle^\cX_{g, n, d},
\]
where the $Q^d$ are formal Novikov variables used to guarantee convergence.

Let $\square$ denote a theory---either the Gromov--Witten theory of a space $\cX$ or the FJRW theory of a quasihomogeneous polynomial $(W,G)$---with state space $\left(H^\square, \langle - ,- \rangle_\square\right)$ with basis $\set{\beta_i}_{i\in I}$ and invariants 
\[
\langle \psi^{a_1} \beta_{i_1}, \ldots,  \psi^{a_n} \beta_{i_n} \rangle^\square_{g, n}.
\]
  
We may define formal generating functions of $\square$--invariants. Let $\bt = \sum_{i \in I} t^i \beta_i$ represent a point of $H^\square$ written in terms of the basis.  
For notational convenience denote the formal series $\sum_{k \geq 0} \bt_k \psi^k $ as $\bt(\psi)$.  
Define the genus $g$ generating function by
\[
\cF_g^\square := \sum_n \frac{1}{n !} \langle \bt(\psi), \ldots,  \bt(\psi) \rangle^\square_{g, n}.
\] 
Let $\cD$ denote the \emph{total genus descendent potential},
\[ 
\cD^\square := \exp \left(\sum_{g \geq 0} \hbar^{g-1} \cF_g^\square\right).
\]

As in Gromov--Witten theory, the correlators in FJRW theory satisfy the so--called \emph{string equation} (SE), \emph{dilation equation} (DE), and \emph{topological recursion relation} (TRR) (For the proof in orbifold Gromov--Witten theory see \cite{Tseng}, in the case of FJRW theory see \cite{FJR1}). 
These equations can be formulated in terms of differential equations satisfied by the various genus $g$ generating functions $\cF^\square_g$.   
We can use this extra structure to rephrase the theory in terms of Givental's \emph{overruled Lagrangian cone}.  For a more detailed exposition of what follows we refer the reader to Givental's original paper on the subject (\cite{G1}).

Let $\sV^\square$ denote the vector space $H^\square((z^{-1}))$, equipped with the symplectic pairing 
\begin{equation}\label{e:residue}
\Omega_\square(f_1, f_2) := \Res_{z=0}\langle f_1(-z), f_2(z)\rangle_\square.
\end{equation}
$\sV^\square$ admits a natural polarization $\sV^\square = \sV^\square_+ \oplus \sV^\square_-$ defined in terms of powers of $z$: 
\begin{align*}\sV^\square_+ &= H^\square[z], \\ \sV^\square_- &= z^{-1}H^\square[[z^{-1}]].\end{align*} 
We obtain Darboux coordinates  $\set{q_k^i, p_{k,i}}$ with respect to the polarization on $\sV^\square$ by representing each 
element of $\sV^\square$ in the form 
\[
\sum_{k \geq 0}\sum_{i \in I} q_k^i \beta_i z^k + \sum_{k \geq 0}\sum_{i \in I} p_{k,i}\beta^i (-z)^{-k-1}
\]
One can view $\cF^\square_0$ as the generating function of a Lagrangian subspace $\sL^\square$ of $\sV^\square$.  Let $\beta_0$ denote the unit in $H^\square$, and make the change of variables (the so--called Dilaton shift)
\[
q_1^0=t_1^0-1 \quad q_k^i=t_k^i \text{ for }(k,i)\neq (1,0).
\]
Then the set 
\[
\sL^\square :=\set{\bp =d_\bq \cF^\square_0}
\]
defines a Lagrangian subspace.  More explicitly, $\sL^\square$ contains the points of the form
\[
-\beta_0 z+ \sum_{\substack{k \geq 0 \\ i \in I}} t_k^i \beta_i z^k +\sum_{\substack{a_1, \ldots , a_n,  a\geq 0 \\ i_1, \ldots , i_n, i \in I}} \frac{t^{i_1}_{a_1}\cdots t^{i_n}_{a_n}}{n!(-z)^{a+1}}\langle\psi^{a_1}\beta_{i_1},\dots,\psi^{a_n}\beta_{i_n}, \psi^a\beta_i \rangle_{0, n+1}^\square\beta^i.
\]
Because $\cF^\square_0$ satisfies the
SE, DE, and TRR, $\sL^\square$ will take a special form.  In fact, $\sL^\square$ is a cone satisfying the condition that for all $f \in \sV^\square$,
\begin{equation}\label{e:overruled}\sL^\square \cap L_f = zL_f\end{equation} 
where $L_f$ is the tangent space to $\sL^\square$ at $f$. Equation~\eqref{e:overruled} justifies the term overruled, as each tangent space $L_f$ is  filtered by powers of $z$:
\[
L_f \supset zL_f \supset z^2L_f \supset \cdots
\] 
and $\sL^\square$ itself is ruled by the various $zL_f$.
The codimension of $zL_f$ in $L_f$ is equal to $\dim(H^\square)$.  

A generic slice of $\sL^\square$ parametrized by $H^\square$, i.e. 
\[
\{f(\bt)| \bt \in H^\square\} \subset \sL^\square,
\]
will be transverse to the ruling. Given such a slice, we can reconstruct $\sL^\square$ as
\begin{equation}\label{e:generic}
\sL^\square = \set{zL_{f(\bt)}| \bt \in H^\square}.
\end{equation}

Givental's $J$--function is defined in terms of the intersection
\[
\sL^\square \cap -\beta_0z \oplus H \oplus \sV^{-}.
\]
Writing things out explicitly, the $J$--function is given by
\[
J^\square(\bt, z) = \beta_0z + \bt + \sum_{n \geq 0} \sum_{\substack{ a \geq 0 \\ h \in I}} \frac{1}{n!z^{a+1}} \langle \bt, \ldots, \bt, \beta_h \psi^a \rangle^\square_{0, n+1} \beta^h.
\] 
In other words, we can obtain the $J$--function by setting $t_k^i=0$ whenever $k>0$. 

In \cite{G3} it is shown that the image of $J^\square(\bt, -z)$ is transverse to the ruling of $\sL^\square$, so $J^\square(\bt, -z)$ is a function satisfying \eqref{e:generic}.  Thus the ruling at $J^\square(\bt, -z)$ is spanned by the derivatives of $J^\square$, i.e.

\begin{equation}\label{e:Jgens}zL_{J^\square(\bt, -z)} = \big\{J^\square(\bt, -z) + z\sum c_i(z) \frac{\partial}{\partial t^i} J^\square(\bt, -z) | c_i(z) \in \CC[z]\big\}.
\end{equation}

By the string equation, $ z\frac{\partial}{\partial t^0} J^\square(\bt, z) = J^\square(\bt, z)$, so \eqref{e:Jgens} simplifies to
\[
zL_{J^\square(\bt, -z)} = \{ z\sum c_i(z) \frac{\partial}{\partial t^i} J^\square(\bt, -z) | c_i(z) \in \CC[z]\big\}.
\]

\subsection{The conjecture}
The LG/CY correspondence was first proposed by physicists (\cite{VW, W}), and is given as a conjecture in \cite{ChR}.  It is phrased mathematically as a correspondence between Gromov--Witten invariants of a Calabi--Yau manifold, and the FJRW invariants of a specified pair $(W,G)$.  In genus 0, the correspondence can be interpreted in terms of the Lagrangian cones of the respective theories.  In \cite{ChR} the genus 0 conjecture is proven for the Fermat quintic using this interpretation.  For simplicity we state the conjecture below only in the particular case of the \emph{mirror quintic}.

In what follows we will use $(W,G)$ in place of $\square$ to denote the FJRW theory of $(W,G)$ and $\cW$ in place of $\square$ to denote the Gromov--Witten theory of $\cW$.  The FJRW and Gromov--Witten state spaces will be $\sH^{nar}_{W,G}$ and $H^{even}_{CR}(\cW)$ respectively.  The full LG/CY correspondence may be stated as a relationship between $\cD^{(W,G)}$ and the analytic continuation of $\cD^{\cW}|_{Q^d = 1}$, where the latter represents the total genus descendant potential for $\cW$ after setting the Novikov variable to one.  Once  Novikov variables have been set to one, the conjecture may be phrased as follows:

\begin{conjecture}[\cite{ChR}]\label{conj} 
Let $\sV^{(W,G)}$ and $\sV^{\cW}$ be the Givental spaces corresponding to the FJRW theory of $(W,G)$ and the Gromov--Witten theory of $\cW$.
\begin{enumerate}
\item There is a degree--preserving $\CC[z, z^{-1}]$--valued linear symplectic isomorphism \[\UU: \sV^{(W,G)} \to \sV^{\cW}\] and a choice of analytic continuation of $\sL^{\cW}$ such that \[\UU(\sL^{(W,G)}) = \sL^{\cW}.\]
\item After analytic continuation, up to an overall constant the total potential functions are related by quantization of $\UU$, i.e.
\[\cD^{\cW} = \hat{\UU}(\cD^{(W,G)}).\]
\end{enumerate}
\end{conjecture}

\begin{remark}\label{conv1} It is not guaranteed that $\cD^{\cW}|_{Q^d = 1}$ is an analytic function. Implicit in the conjecture, however, is the claim that after setting the Novikov variables to one, $\cD^{\cW}$ converges in some neighborhood.  Thus one must first check convergence in order to prove the LG/CY correspondence.  For the purposes of this paper however, the necessary convergence will follow from the mirror theorem of \cite{LSh} restated here in equation~\eqref{e:A-model}.
\end{remark}

\subsubsection{The Small Slice of $\sL$}\label{ss:slice}

In \cite{ChR}, the LG/CY correspondence is proven by relating the respective $J$--functions for the two theories.  A crucial point in the argument is that in the case of the quintic three--fold $M$, the $J$--function $J^{M}(\bs, z)$ (and hence the full Lagrangian cone $\sL^{M}$) may be recovered from the small $J$--function \[J^{M}_{small}(s, z) := J^{M}(\bs, z)|_{\bs = s \in H^2(M)}.\]  This is no longer true for the mirror quintic. 

Although calculating the big $J$--function for $\cW$ appears to be a difficult problem, in \cite{LSh} its derivatives $\frac{\partial}{\partial s^i} J^{\cW}(\bs, z)$ may be calculated at any point $sH \in H^2(\cW)$.  This allows us to prove a ``small'' version of the LG/CY correspondence for the mirror quintic.  We will phrase the theorem in analogy with Conjecture~\ref{conj}.

In order to do so we define the \emph{small slice} of $\sL^{\cW}$ and $\sL^{(W,G)}$ to be that part of the ruling  coming from $sH \in H^2(\cW)$ and $t\phi_{\jw^2} \in \sH^2_{W,G}$ respectively:
\begin{definition}\label{d:small} The small slices of $\sL^{\cW}$ and $\sL^{(W,G)}$ are defined by
\[
 \sL^{\cW}_{small} :=  \{zL_{J^{\cW}(\bs, -z)} | \bs  = sH\}.
\]
\[
 \sL^{(W,G)}_{small} :=  \{zL_{J^{(W,G)}(\bt, -z)}| \bt =t\phi_{\jw^2}\}.
\]
\end{definition}

Our main theorem may then be stated as a correspondence between the small slices of the Lagrangian cones $\sL^{(W,G)}$ and $\sL^{\cW}$.

\begin{theorem}(=Theorem~\ref{t:main})
There exists a symplectic transformation $\UU$ identifying the analytic continuation of  $\sL^{\cW}_{small}$ with $\sL^{(W,G)}_{small}$.
\end{theorem}

%% file: twistedFJRW.tex
\section{Twisted theory}\label{twisted}
In this section we compute the FJRW invariants necessary to prove the correspondence.
Fix as a basis for $\sH_{W,G}^{nar}$ the set $\set{\phi_h}_{h\in \hat{S}}$ defined in Section~\ref{s:FJRWinv}. 

We will construct a \emph{twisted} FJRW theory whose invariants coincide with those of $(W,G)$ in genus zero.
We first extend the state space
\[
\sH_{W,G}^{ext}:=\sH_{W,G}^{nar}\oplus\bigoplus_{h\in G \setminus \hat{S}} \phi_h \CC.
\]
Any point $\bt\in\sH_{W,G}^{ext}$ can be written as $\bt=\displaystyle \sum_{h \in G}t^h\phi_h$. 
Let $i_k(h) := \langle \varTheta_{k}(h)-\tfrac15 \rangle$, where $\langle - \rangle$ denotes the fractional part. Notice $i_k(h)=\tfrac 45$ exactly when $\varTheta_k(h)=0$. 
Set 
\[
\deg_W(\phi_h) := 2\sum_{k=1}^5 i_k(h).
\]
For $h \in \hat{S}$, this definition matches the W--degree defined in \eqref{e:degw}.

We extend the definition of our FJRW invariants to include insertions $\phi_h$ in $\sH_{W,G}^{ext}$.  Namely, set 
\[
\br{\psi^{a_1}\phi_{h_1},\dots, \psi^{a_n}\phi_{h_n}}_{0,n}^{(W,G)} = 0
\] 
if $h_i \in G \setminus \hat{S}$.

We would like to unify our definition of the extended FJRW invariants, by re--expressing them as integrals over $ (\widetilde A_4^4)_{0,n}$, a slight variation of $ (A_4^4)_{0,n}$.
We will make use of the following lemma.

\begin{lemma}[\cite{ChR}]\label{l:linebndles}
Let $\cC$ be a $d$--stable curve and let $M$ be a line bundle pulled back from the coarse space. If $l|d$, there is an equivalence between two categories of $l$th roots $\cL$ on $d$--stable curves:
\[
\set{\cL|\cL^{\otimes l}\cong M}\leftrightarrow \bigsqcup_{0\leq E<  \sum lD_i}\set{\cL|\cL^{\otimes l}\cong M(-E), \mult_{p_i}(\cL)=0}.
\]
where the union is taken over divisors $E$ which are linear combinations of integer divisors $D_i$ corresponding to the marked points $p_i$. 
\end{lemma}

\begin{proof}
Let $p$ denote the map which forgets stabilizers along the markings.  The correspondence is simply $\cL \mapsto p^* p_*(\cL)$.  
\end{proof}

\begin{definition}For $m_1,\dots, m_n\in \{\frac{1}{5}, \frac{2}{5}, \frac{3}{5}, \frac{4}{5}, 1\}$, consider the stack $\tilde A_4\left(m_1,\dots,m_n\right)_{g,n}$ classifying genus $g$, $n$--pointed, 5--stable curves equipped with fifth roots:
\[
\widetilde A_4\left(m_1,\dots,m_n\right)_{g,n}:=\set{(\cC, p_1,\dots, p_n, \cL, \varphi)|\phi:\cL^{\otimes 5}\stackrel{\sim}{\to} \omega_{\log}(-\sum_{i=1}^n5m_iD_i),\; \mult_{p_i}(\cL)=0}, 
\]
where the integer divisors $D_i$ correspond to the markings $p_i$. 
\end{definition}
The moduli space $\widetilde A_4\left(m_1,\dots,m_n\right)_{g,n}$ also has a universal curve $\sC\to \widetilde{A}_4$ and a universal line bundle $\widetilde{\LL}$. 

We now define an analogue of $(A_4^4)_{g,n}$, replacing $(A_4)_{g,n}$ with $(\widetilde{A}_4)_{g,n}$ in each factor.  
For $1 \leq i \leq n$, let $m_i = (m_{1i}, \ldots, m_{5i})$ be a $5$-tuple of fractions satisfying $m_{ki} \in \{\frac{1}{5}, \frac{2}{5}, \frac{3}{5}, \frac{4}{5}, 1\}$, and $\langle \sum_{k=1}^5 m_{ki} \rangle = 0$.  Let $\bv{m}$ denote the $5 \times n$ matrix $(\bv{m})_{ki} = m_{ki}$.

Define
\[
\widetilde{A}_4^4(\bv{m})_{g,n} := \widetilde A_4(m_{11},\dots,m_{1n})_{g,n}\times_{\cM_{g,n,5}}\dots \times _{\cM_{g,n,5}} \widetilde A_4(m_{41},\dots,m_{4n})_{g,n}.
\]
$\widetilde{A}_4^4(\bv{m})_{g,n}$ carries four universal line bundles $\widetilde{\LL}_1,\dots,\widetilde{\LL}_4$ satisfying 
\[
(\widetilde{\LL}_k)^{\otimes 5} \cong \omega_{\log}\left(- \sum_{i=1}^n 5m_{ki}D_i\right).
\] 
Define a fifth line bundle
\[
\widetilde{\LL}_5:=\widetilde{\LL}_1^\vee\otimes\dots\otimes\widetilde{\LL}_4^\vee\otimes \omega_{\log}\left( - \sum_{i=1}^n \sum_{k=1}^5m_{ki}D_i\right).
\] 
One can check that $(\widetilde {\LL}_5)^{\otimes 5}\cong \omega_{\log}(- \sum_{i=1}^n 5m_{5i}D_i)$.   
 
The above moduli space yields a uniform way of defining the extended FJRW invariants for $(W,G)$.  Given $\phi_{h_1}, \ldots, \phi_{h_n} \in \sH_{W,G}^{ext}$, let 
\begin{equation*}
I(\bv{h}) =\left(\begin{matrix}
i_1(h_1)+\frac15 & \cdots & i_1(h_n)+\frac15\\
\vdots & & \vdots\\
i_5(h_1)+\frac15 & \cdots & i_5(h_n)+\frac15
\end{matrix}\right).
\end{equation*}
Consider the following proposition.

\begin{prop}\label{p:chern}  
On $\widetilde{A}_4^4(I(\bv{h}))_{0,n}$,  $\pi_*\big(\bigoplus_{k=1}^5\widetilde{\LL}_k\big)$ vanishes and $R^1\pi_*\big(\bigoplus_{k=1}^5\widetilde{\LL}_k\big)$ is locally free.  Furthermore, 
\begin{equation}\label{e:calc12}
\br{\psi^{a_1}\phi_{h_1},\dots, \psi^{a_n}\phi_{h_n}}_{0,n}^{(W,G)} =625 \int_{\tilde{A}^4_4(I(\bv{h}))_{0,n}} \prod \psi_i^{a_i}\cup c_{top}\Big(R^1\pi_*\big(\bigoplus_{k=1}^5\widetilde{\LL}_k\big)^\vee\Big).
\end{equation}
\end{prop}

\begin{proof}
Comparing $A_4$ and $\widetilde{A}_4$, we see that if $m_{ki}\in \set{\frac15, \frac 25, \frac35 , \frac45}$ for all $k, i$, we can identify $\widetilde{A}_4^4(\bv{m})_{g,n}$ with $A_4^4(\bv{m})_{g,n}$  via Lemma~\ref{l:linebndles}.  Under this identification $R^j\pi_*(\widetilde{\LL}_k) = R^j\pi_*(\LL_k)$.  This gives \eqref{e:calc12} in the case $\phi_{h_1}, \ldots, \phi_{h_n} \in \sH^{nar}_{W,G}$.

To finish the proof we must consider the case where $h_i \in G \setminus \hat{S}$ for some $i$.  In this case $(I(\bv{h}))_{ki} = 5$ for some $k$.  Thus it suffices to prove that if $m_{ki} = 5$ for some $i$ and $k$, then
$\pi_*\big(\bigoplus_{k=1}^5\widetilde{\LL}_k\big) = 0$
and
$c_{top}\Big(R^1\pi_*\big(\bigoplus_{k=1}^5\widetilde{\LL}_k\big)\Big)=0.$

Without loss of generality assume $m_{k1}=5$. Consider the integer divisor $D_1$ on $\widetilde{A}_4^4(\bv{m})_{0,n}$ corresponding to the first marked point. We have the following exact sequence
\[
0\to \widetilde\LL_k \to \widetilde\LL_k(D_1)\to \widetilde\LL_k(D_1)|_{D_1}\to 0.
\]
This gives rise to the long exact sequence
\begin{align*}
0 &\to \pi_*(\widetilde\LL_k) \to \pi_*(\widetilde\LL_k(D_1)) \to \pi_*(\widetilde\LL_k(D_1)|_{D_1})\\
 &\to R^1\pi_*(\widetilde\LL_k) \to R^1\pi_*(\widetilde\LL_k(D_1))\to R^1\pi_*(\widetilde\LL_k(D_1)|_{D_1})\to 0.
\end{align*}

The first two terms are 0. Indeed, consider first $\pi_*(\widetilde{\LL}_k)$. The fiber over the point $(\cC,p_1,\dots,p_n,\widetilde \cL_1,\dots, \widetilde\cL_5)$ is equal to $H^0(\cC,\widetilde{\cL}_k)$. 
As in Proposition~\ref{p:concave} we will show that $\widetilde \cL_k$ has no global sections by computing its degree on each irreducible component of $\cC$.  If $\cC$ is irreducible, $\deg(\widetilde\cL_k) < 0$ and the claim follows.  If not, let $\Gamma$ denote the dual graph to $\cC$, let $v$ be a vertex corresponding to the irreducible component $\cC_v$ and let $P_v$ be the set of special points of $\cC_v$.  As in Proposition~\ref{p:concave}, we obtain the inequality $\deg( \widetilde\cL_k|_{\cC_v} )<\#\text{nodes}(\cC_v) -1$.  Again one can proceed vertex by vertex starting from outer vertices of $\Gamma$ and show that the restriction of $\widetilde\cL_k$ to each component has no nonzero global sections.

We can do the same with $\pi_*(\widetilde{\LL}_k(D_1))$, with one alteration. If $\cC$ is reducible, and $v '$ corresponds to the irreducible component carrying the first marked point, then $\deg \widetilde \cL_k(D_1)|_{\cC_{v '}}<\#\text{nodes}(\cC_{ v '})$.  But any section of $\widetilde \cL_k(D_1)$ must still vanish on all other components of $C$, and by degree considerations it must therefore vanish on $\cC_{v '}$.

$D_1$ is zero--dimensional on each fiber, so $R^1\pi_*(\widetilde\LL_k(D_1)|_{D_1})$ also vanishes.  The above long exact sequence above becomes
\[
0 \to \pi_*\widetilde\LL_k(D_1)|_{D_1}\to R^1\pi_*\widetilde\LL_k\to R^1\pi_*\widetilde\LL_k(D_1)\to 0. 
\]
Therefore 
\[
c_{top}(R^1\pi_*\widetilde\LL_k)=c_{top}(\pi_*\widetilde\LL_k(D_1)|_{D_1})\cdot c_{top}(R^1\pi_*\widetilde\LL_k(D_1)). 
\]
But $c_{top}(\pi_*\widetilde\LL_k(D_1)|_{D_1})=0$, as $\widetilde{\LL}_k(D_1)|_{D_1} \cong \LL_k|_{D_1}$ is a fifth root of $\omega_{\log}|_{D_1}$ which is trivial. Thus $c_{top}(R^1\pi_*\widetilde\LL_k)=0$ as well. 
\end{proof}

We may define a $\CC^*$--equivariant generalization of the above theory.  This will allow us to compute invariants which, in the non--equivariant limit coincide with the genus zero FJRW invariants above.
Given a point $(\cC,p_1,\dots,p_n,\widetilde \cL_1,\dots, \widetilde\cL_5)$ in $(\widetilde{A}_4^4)_{g,n}$, let $\CC^*$ act on the total space of $\bigoplus_{k=1}^5 \widetilde{\cL}_k$ by multiplication of the fiber.  This induces an action on $(\widetilde{A}_4^4)_{g,n}$.  

Set $R=H^*_{\CC^*}(pt,\CC)[[s_0,s_1, \dots]]$, the ring of power series in the variables $s_0,s_1,\dots$ with coefficients in the equivariant cohomology of a point,  $H^*_{\CC^*}(pt, \CC)=\CC[\lambda]$. 
Define a multiplicative characteristic class $\bc$ taking values in $R$, by
\[
\bc(E):=\exp\left(\sum_k s_k \ch_k(E) \right)
\]
for $E\in K^*((\widetilde A_4^4)_{g,n})$. 

Define the twisted state space
\[
\sH^{tw}:=\sH_{W,G}^{ext}\otimes R \cong \bigoplus_{h\in G} R\cdot \phi_h
\] 
and extend the pairing by 
\begin{equation}\label{e:pairing}
\br{\phi_{h_1},\phi_{h_2}}:=\begin{cases}
    \displaystyle\prod_{\{k\,| i_k(h_1)=4/5\}}\exp(-s_0) & \text{ if } h_1=(h_2)^{-1}\\
    0 & \text{otherwise}.
    \end{cases}
\end{equation}
In this definition, the empty product is understood to be 1. 

We define the symplectic vector space $\sV^{tw}:=\sH_{tw}((z^{-1}))$, with the symplectic pairing defined as in equation~\eqref{e:residue}.

We may also define \emph{twisted correlators} as follows.  Given $\phi_{h_1}, \ldots, \phi_{h_n}$ basis elements in $\sH^{tw}$, define the invariant
\[
\br{\psi^{a_1}\phi_{h_1},\dots,\psi^{a_n} \phi_{h_n}}_{g,n}^{tw} :=625 \int_{\tilde{A}^4_4(I(\bv{h}))_{g,n}} \prod \psi_i^{a_i}\cup \bv{c}\Big(R\pi_*\big(\bigoplus_{k=1}^5\widetilde{\LL_k}\big)\Big).
\]
taking values in $R$.  We can organize these invariants into generating functions $\cF_g^{tw}$ and $\cD^{tw}$ as in Section~\ref{formal}.

Specializing to particular values of $s_d$ yield different twisted invariants.  In particular, if $s_d=0$ for all $d$, we get what is referred to as the \emph{untwisted theory}. We will denote the generating functions of the untwisted theory by $\cF_g^{un}$ and $\cD^{un}$.

On the other hand, setting 
\begin{equation}\label{e:sd}
s_d=\begin{cases}
-\ln \lambda & \text{if } d=0\\
\frac{(d-1)!}{\lambda^d} & \text{otherwise}
\end{cases}
\end{equation} 
we obtain the (extended) FJRW--theory invariants defined above.  To see this first consider the following lemma.

\begin{lemma}\cite[Lemma 4.1.2]{ChR}\label{l:ctop}
With $s_d$ defined as in \eqref{e:sd}, the multiplicative class $\bc(-V)=e_{\CC^*}(V^\vee)$. In particular, the non--equivariant limit yields the top chern class of $V^\vee$.
\end{lemma}

\begin{proof}
We can check this on a line bundle, and then apply the splitting principle. Consider a line bundle $\cL$. Then we have 
\begin{align*}
\exp\left(\sum_{d\geq 0} s_d\ch_d(-\cL)\right) &= \exp\left(\ln \lambda\ch_0(\cL) -\sum_{d>0} s_d\ch_d(\cL)\right)\\
 &=\exp\left(\ln \lambda\ch_0(\cL^\vee) -\sum_{d>0} (-1)^ds_d\ch_d(\cL^\vee)\right)\\
 &=\exp\left(\ln \lambda\ch_0(\cL^\vee) +\sum_{d>0} (-1)^{d-1}\tfrac{(d-1)!}{\lambda^d}\ch_d(\cL^\vee)\right)\\
 &=\lambda\exp\left(\sum_{d>0} (-1)^{d-1}\tfrac{c_1(\cL^\vee)^d}{d\lambda^d}\right)\\
 &=\lambda\exp\left(\ln(1+\frac{c_1(\cL^\vee)}{\lambda}\right)\\
 &=\lambda+c_1(\cL^\vee)
\end{align*}
\end{proof}

By Proposition~\ref{p:concave}, $\pi_*(\bigoplus \widetilde{\LL}_k)=0$ and $\bc(R\pi_*(\widetilde{\LL}_k))=\bc(-R^1\pi_*(\widetilde{\LL}_k))$. Setting $s_d$ as in \eqref{e:sd} therefore yields 
\[
\bc\Big(R\pi_*\Big(\bigoplus_{k=1}^5\widetilde{\LL}_k\Big)\Big) = e_{\CC^*}\Big(R^1\pi_*\Big(\bigoplus_{k=1}^5\widetilde{\LL}_k\Big)^\vee\Big).
\] 
Applying Proposition~\ref{p:chern} we obtain the following
\begin{corollary}\label{c:nonequivlimit}
After specializing $s_d$ to the values in \eqref{e:sd}, 
\[
\lim_{\lambda\to 0} \cF_0^{tw}=\cF_0^{(W,G)}.
\]
\end{corollary}

We will compute twisted invariants by relating them to untwisted invariants, which we can compute directly.
As before it is easy to check that $\cF_0^{un}$ satisfies SE, DE, and TRR, (where $\phi_\jw$ plays the role of the unit in this theory, as in Remark~\ref{r:identity element}) so it defines an overruled Lagrangian cone $\sL^{un}\subset \sV^{un}$, satisfying the same geometric properties as described in Section~\ref{formal}. We obtain the untwisted $J$--function
\[
J^{un}(\bv{t},-z)= -z\phi_\jw + \bt + \sum_{n \geq 0}\sum_{\substack{ a \geq 0 \\ h \in G}} \frac{1}{n!(-z)^{a+1}}\br{ \bv{t}, \dots \bv{t},\psi^{a}\phi_h }^{un}_{0,n+1}\phi^{h}.
\]
We may similarly define $J^{tw}(\bt, z)$ and $\sL^{tw}$ in terms of $\sF^{tw}_0$, but it is not obvious $\sL^{tw}$ is a Lagrangian cone. Rather than proving this directly, we will use the methods of quantization. Let $B_d(x)$ denote the $d$th Bernoulli polynomial, and recall $i_k(h) = \langle \varTheta_{k}(h)-\tfrac15 \rangle$.
\begin{prop}\label{p:sympl}
The symplectic transformation
\[
\Delta=\bigoplus_{h\in G} \prod_{k=1}^5 \exp\left(\sum_{d\geq 0} s_d\frac{B_{d+1}\big(i_k(h)+\tfrac{1}{5}\big)}{(d+1)!}z^d\right) 
\] 
satisfies $\sL^{tw}=\Delta(\sL^{un})$.
\end{prop}

\begin{proof}Note first that the identity 
$B_d(1-x)=(-1)^{d}B_d(x)$ implies $\Delta$ is symplectic. 

The proof is the same as the proof in \cite{ChR} and \cite{ChZ}, with some slight modification. We give a sketch here.  The strategy is to first relate $\cD^{un}$ to $\cD^{tw}$ via the quantization $\hat\Delta$.  The desired statement then follows by taking the semiclassical limit (see \cite{CPS}).
 
We will prove that
\begin{equation}\label{e:quant}
\hat\Delta\cD^{un}=\cD^{tw}
\end{equation} 
by viewing both sides as functions with respect to the variables $s_d$ and showing they are both solutions to the same system of differential equations.
First notice that both sides of \eqref{e:quant} have the same initial condition, i.e. when $\bv{s}=0$ they are equal. We will show that $\cD^{tw}$ and $\cD^{un}$ both satisfy
\begin{equation}\label{e:diffeq}
\frac{\partial \Phi}{\partial s_d}=\sum_{k=0}^5 P_d^{(k)}\Phi
\end{equation}
where 
\begin{align*}
P_d^{(k)}=\frac{B_{d+1}(\tfrac{1}{5})}{(d+1)!} \frac{\partial}{\partial t^\jw_{d+1}}-&\sum_{\stackrel{a \geq 0}{h \in G}} \frac{B_{d+1}(i_k(h)+\tfrac{1}{5})}{(d+1)!}t^h_a\frac{\partial}{\partial t^h_{a+d}}\\ +&\frac{\hbar}{2}\sum_{\stackrel{a+a'=d-1}{h, h' \in G}} (-1)^{a'}\eta^{h,h'}\frac{B_{d+1}(i_k(h)+\tfrac{1}{5})}{(d+1)!} \frac{\partial^2}{\partial t^h_a \partial t^{h'}_{a'}},
\end{align*}
and $\eta^{h,h'}$ denotes the inverse pairing. 

It is clear that $\hat\Delta\cD^{un}$ satisfies the equation, it remains to show that $\cD^{tw}$ does as well. 
Substituting $\cD^{tw}$ for $\Phi$ in \eqref{e:diffeq} and taking the derivative with respect to $s_d$, we see that the equation reduces to
\begin{align*}
\sum_{n\geq 0} \frac{1}{n!}\Big\langle\bt(\psi), \dots, \bt(\psi); &\ch_d(R\pi_*(\widetilde{\LL}_k))\cdot \bv{c}(R\pi_*(\bigoplus_l\widetilde{\LL}_l))\Big\rangle_{0,n} \\
&=P_d^{(k)}\cF_g^{tw}+\frac{\hbar}{2}\sum_{\stackrel{a+a'=d-1}{h, h' \in G}}(-1)^{a'}\eta^{h,h'}\frac{B_{d+1}(i_k(h)+\tfrac{1}{5})}{(d+1)!}\frac{\partial \cF_g^{tw}}{\partial t^h_a}\frac{\partial \cF_g^{tw}}{\partial t^{h'}_{a'}}
\end{align*}

This equation was proven in \cite{ChZ}, and generalized to the extended state space in \cite{ChR}. It is proved using Grothendieck--Riemann--Roch to give an expression for $\ch_d(R\pi_*(\widetilde{\LL}_k))$. 
\end{proof}

It will be useful to separate the summands of $J^{un}(\bt, z)$ in terms of powers of $t^h$.
Given a function $\bv{n}: G \to \ZZ_{\geq 0}$, let 
$J^{un}_{\bv{n}}(\bt, z)$ denote the $\prod_{h \in G}(t^h)^{\bv{n}(h)}$--summand of $J^{un}(\bt,z)$.
Proposition~\ref{p:1.3} plus a straightforward $\psi$--class calculation shows that the correlator $\br{\phi_{h_1},\dots, \phi_{h_{n}},\psi^l\phi_{h}}_{0,n+1}^{un}=1$ 
when $i_k(h)= \langle \tfrac35-\sum_m i_k(h_m)\rangle$ and $l=n-2$. It is zero otherwise.  Furthermore, $i_k(h^{-1}) = \langle \tfrac35 - i_k(h)\rangle$. We arrive at the following pleasant formula
\[
J^{un}_{\bv{n}}(\bt,z) = \frac{\prod_{h \in G}(t^h)^{\bv{n}(h)} }{z^{|\bv{n}|-1}\prod_{h \in G}\bv{n}(h)!}\phi_{h_{\bv{n}}},
\]
with $h_{\bv{n}}$ defined by $i_k(h_\bv{n}) = \langle \sum_{h\in G} \bv{n}(h)i_k(h) \rangle$.

We conclude that
\begin{equation}\label{e:jun}
J^{un}(\bt,z)=\sum_{\bv{n}} \frac{\prod_{h \in G}(t^h)^{\bv{n}(h)} }{z^{|\bv{n}|-1}\prod_{h \in G}\bv{n}(h)!}\phi_{h_{\bv{n}}}.
\end{equation}
Proposition~\ref{p:sympl} allows us to describe $\sL^{tw}$ in terms of $\sL^{un}$.  Combining this with 
Equation~\eqref{e:jun}, we will obtain an explicit description of a slice of $\sL^{tw}$.  This will then determine $J^{tw}(\bt, z)$.

Define $D_h=t^h\tfrac{\partial}{\partial t_0^h}$, and put $D^k=\sum_{h \in G} i_k(h)D_h$. Notice that $D_h J^{un}_{\bv{n}}(\bt,z)= \bv{n}(h)J^{un}_{\bv{n}}(\bt,z)$. 
Consider the following functions:
\begin{align*}
\bv{s}(x)&=\sum_{d\geq 0} s_d\frac{x^d}{d!}\\
G_y(x,z)&=\sum_{l,m\geq 0} s_{l+m-1} \frac{B_m(y)}{m!}\frac{x^l}{l!}z^{m-1}.
\end{align*}
These functions satisfy the following:
\begin{align*}
G_y(x,z)&=G_0(x+yz,z)\\
G_0(x+z,z) &=G_0(x,z)+\bv{s}(x)
\end{align*}

\begin{prop}
The slice defined by
\[
J^{\mathbf{s}}(\bt, z)=\prod_{k=1}^5\left(\exp(-G_{1/5}(zD^k,z)\right)J^{un}(\bt,z)
\]
lies on $ \sL^{un}$. 
\end{prop}

\begin{proof}
This lemma appears in \cite{CCIT} and \cite{ChR}.  We give it again here for the purpose of completeness. 
Any element $f \in \sV^{tw}$ can be written in the form 
\[
f=-z\phi_\jw+\sum_{l \geq 0}\bt_lz^{l} +\sum_{l \geq 0}\frac{\mathbf{p}_l(f)}{(-z)^{l+1}} 
\]
for some $\bv{p}_l(f)=\sum_{h \in G} p_{l,h}(f)\phi^h$.
If $f \in \sL^{un}$, then we know
\[
\bv{p}_l(f)=\sum_{n\geq 0}\sum_{h\in G}\frac{1}{n!} \br{\bt(\psi),\dots,\bt(\psi),\psi^l\phi_h}_{0,n+1}^{un}\phi^h
\]

The idea is to define 
\[
E_l(f)=\mathbf{p}_l(f)-\sum_{n\geq 0} \sum_{h \in G}\frac{1}{n!} \br{\bt(\psi),\dots,\bt(\psi),\psi^l\phi_h}_{0,n+1}^{un}\phi^h
\]
and show that $E_l(J^{\mathbf{s}})=0$. 

Let $\deg s_d= d+1$, and proceed by induction on the degree. Since $J^{un}(\bt,z)$ lies on $\sL^{un}$, the degree zero terms of $E_l(J^{\mathbf{s}})$ vanish. Now assuming the degree $n$ terms vanish, we will show that the degree $n+1$ terms vanish. Because of the vanishing up to degree $n$, there exists another family $\widetilde J^{\mathbf{s}}(\bt,-z)$ such that $E_l(J^{\mathbf{s}})$ and $E_l(\widetilde{J}^{\mathbf{s}})$ agree up to degree $n$. Differentiating, we obtain
\[
\frac{\partial}{\partial s_d}E_l(J_{\mathbf{s}})=d_{J^{\mathbf{s}}}E_j(z^{-1}P_dJ_{\mathbf{s}}) 
\]
where 
\[
P_d=\sum_{k=1}^5\sum_{m=0}^{d+1} \frac{1}{m!(d+1-m)!}z^mB_m(\tfrac{1}{5})(zD^k)^{d+1-m}.
\]

Up to degree $n$, the right hand side coincides with $d_{\widetilde{J}^{\mathbf{s}}}E_l(z^{-1}P_d\widetilde{J}^{\mathbf{s}})$,  which vanishes because the term in parentheses lies on $T_{d_{\widetilde{J}^{\mathbf{s}}}\cL^{un}}$. 
\end{proof}

Applying $\Delta$ to $J^{\bs}(\bt,-z)$ yields a slice of the twisted cone $\sL^{tw}$. To facilitate computation, we express $J^{\bv{s}}(\bt, -z)$ in terms of monomials in the $t^h$ variables
\[
J^{\bv{s}}(\bt,-z)=\sum_{\bv{n}} \prod_{k=1}^5\exp \left(-G_{\frac15}\bigg(\big(\sum_{h\in G}\bv{n}(h)i_k(h)
\big)z,z\bigg)\right)J^{un}_{\bv{n}}(\bt,-z) ,
\]
and express $\Delta$ as 
\[
\Delta=\prod_{k=1}^5\bigoplus_{h\in G}\exp\Big( G_{\frac15}\big(i_k(h)z, 
z\big)\Big).
\]
We can write 
\begin{align*}
\Delta\left(J^{\bv{s}}(\bt,-z)\right)&=\sum_{\bv{n}} \prod_{k=1}^5\exp\left(G_{\frac15}\bigg(\Big\langle \sum_{h\in G} \bv{n}(h)i_k(h)\Big\rangle z,
z\bigg)-G_{\frac15}\bigg( \sum_{h\in G} \bv{n}(h)i_k(h)z,
z\bigg)\right) J^{un}_{\bv{n}}(\bt,z)\\
    & =\sum_{\bv{n}} \prod_{k=1}^5\exp\left(\sum_{0\leq b<\lfloor \sum_{h\in G} \bv{n}(h)i_k(h) \rfloor} -\bs \bigg(\tfrac 15z+\Big\langle \sum_{h\in G} \bv{n}(h)i_k(h)\Big\rangle z+bz\bigg) \right)J^{un}_{\bv{n}}(\bt,z)
\end{align*}

Setting $F_{\bv{n}}=\lfloor \sum_{h\in G} \bv{n}(h)i_k(h) \rfloor$, define the modification factor
\[
M_{\bv{n}}(z)=\prod_{k=1}^5\exp\left(\sum_{0\leq b<F_{\bv{n}}} -\bs \bigg(-\tfrac 15z-\Big\langle \sum_{h\in G} \bv{n}(h)i_k(h)\Big\rangle z-bz\bigg) \right)
\]
Setting $s_d$ as in \eqref{e:sd}, we get 
\begin{align*}
M_{\bv{n}}(z)&=\prod_{\stackrel{1 \leq k \leq 5}{0\leq b<F_{\bv{n}}}} \exp\left(-s_0-\sum_{d>0} s_d\frac{\left(-\tfrac 15 z-\br{ \sum_{h\in G} \bv{n}(h)i_k(h)}z-bz\right)^d}{d!}\right)\\
    &=\prod_{\stackrel{1 \leq k \leq 5}{0\leq b<F_{\bv{n}}}} \exp\left(\ln \lambda +\sum_{d>0}(-1)^{d-1}\frac{\left(\tfrac 15 z +\br{ \sum_{h\in G} \bv{n}(h)i_k(h)}z+ bz\right)^d}{d}\right)\\
    &=\prod_{\stackrel{1 \leq k \leq 5}{0\leq b<F_{\bv{n}}}} \lambda \exp\left( \ln\bigg(1+ \frac{\tfrac 15 z+\br{ \sum_{h\in G} \bv{n}(h)i_k(h)}z+bz}{\lambda} \bigg) \right)\\
    &=\prod_{\stackrel{1 \leq k \leq 5}{0\leq b<F_{\bv{n}}}} \left(\lambda +\tfrac 15 z+\Big\langle \sum_{h\in G} \bv{n}(h)i_k(h)\Big\rangle z+bz \right)
\end{align*}

Define the $I$--function:
\begin{equation}\label{e:modif}
I^{tw}(\bt,z):=\sum_{\bv{n}} M_{\bv{n}}(z)J^{un}_{\bv{n}}(\bt,z) 
\end{equation}

By Proposition~\ref{p:sympl}, $I^{tw}\subset \sL^{tw}$.  
Furthermore, we know by Corollary~\ref{c:nonequivlimit} taking the non--equivariant limit $\lambda \mapsto 0$ recovers the FJRW invariants of $(W,G)$. Define 
\[
I^{(W,G)}(\bt,z):=\lim_{\lambda\to 0} I^{tw}(\bt,z)|_{\bt \in \sH^{nar}_{W,G}}.
\]
By Corollary~\ref{c:nonequivlimit}, the function $I^{(W,G)}(\bt, z)$ lies on
$\sL^{(W,G)}$.  

To state the mirror theorem, we apply the following:
\begin{convention}\label{defy conventions} From this point forward, we restriction to $\bt$ of degree two in $\sH^{nar}_{W,G}$.
Let $t$ denote the dual coordinate to $\phi_{\jw^2}$.  Then we may write
\begin{equation}\label{e:tsmall}
\bt= t \phi_{\jw^2 } + \sum_{\substack{ h \in \hat{S} \setminus \jw^2\\ \deg_W\phi_h = 2}} t^h\phi_h.
\end{equation}
\end{convention}

We will need the following lemma.
\begin{lemma}\label{l:I}  For $\bt$ as in \eqref{e:tsmall}, we may expand the $I$--function as
\begin{equation}
I^{(W,G)}(\bt,z)= zF(\bt)\phi_\jw +\bv{G}(\bt) +\cO(z^{-1})
\end{equation}
with $F(\bt)=F_0(t)+\cO(2)$ and 
\[
\bv{G}(\bt)=G_{\jw^2}(t)\phi_{\jw_2}+\sum_{\substack{ h \in \hat{S} \setminus \jw^2\\ \deg_W\phi_h = 1}} t^hG_h(t)\phi_h +\cO(2).
\]
Here $\cO(2)$ denotes terms of degree at least two in the variables $\{t^h | h\neq \jw^2\}$.
\end{lemma}

\begin{proof}
Applying the non--equivariant limit $\lambda \mapsto 0$ to \eqref{e:modif}, we can write 
\[
I^{(W,G)}(\bt,z)=\sum_{\bv{n}}\prod_{\substack{k=1,\dots,5\\ 0\leq m<\lfloor \sum \bv{n}(h)i_k(h) \rfloor}}\left(\bigg(\Big\langle\sum_{h \in \hat{S}} \bv{n}(h)i_k(h)\Big\rangle+\tfrac{1}{5}+m\bigg)z\right) \frac{\prod_h t^{\bv{n}(h)}}{z^{|\bv{n}|-1}\prod_h \bv{n}(h)!}\phi_{h_{\bv{n}}},
\] 
where the first sum is now over $\bv{n}: \hat{S} \to \ZZ_{\geq 0}$.

For a given $\bv{n}$, the power of $z$ in the corresponding summand is 
\begin{equation*}
1-\sum_{h\in\hat{S}} \bv{n}(h)+\sum_{k=1}^5\lfloor \sum_{h\in\hat{S}} \bv{n}(h)i_k(h)\rfloor
\end{equation*}
where the first two terms are the contribution from $J^{un}(\bt,z)$ and the last sum is from the modification factor $M_{\bv{n}}$. Since we have restricted to $\deg_W(\phi_h)\leq 2$, we have
\[
1-\sum_{h\in\hat{S}} \bv{n}(h)+\sum_{k=1}^5\lfloor \sum_{h\in\hat{S}} \bv{n}(h)i_k(h)\rfloor\leq 1-\sum_{h\in\hat{S}} \bv{n}(h)+\sum_{k=1}^5 \sum_{h\in\hat{S}} \bv{n}(h)i_k(h)\leq 1.
\]

Consider the coefficient of $z^1$. For a particular $\bv{n}$ to contribute to this term, it must be the case that
\[
\sum_{h\in\hat{S}} \bv{n}(h)=\sum_{k=1}^5\lfloor \sum_{h\in\hat{S}} \bv{n}(h)i_k(h)\rfloor
\]
which implies that
\[
\sum_{k=1}^5\lfloor \sum_{h\in\hat{S}} \bv{n}(h)i_k(h)\rfloor=
\sum_{k=1}^5 \sum_{h\in\hat{S}} \bv{n}(h)i_k(h).
\] 
Therefore $ i_k(h_\bv{n})=\br{\sum_{h\in\hat{S}}\bv{n}(h)i_k(h)}=0$ for $1\leq k \leq 5$, and $h_{\bv{n}} = \jw$.
This gives us the first term $zF(\bt)\phi_{\jw}$. It is clear that $F(\bt)=F_0(t)+\cO(2)$, because for $\br{\sum_{h\in\hat{S}}\bv{n}(h)i_k(h)}=0$ to hold for all $k$ there cannot be just one $t^h$ variable.

Now consider the coefficient of $z^0$. There are two kinds of summands we need to consider, those only in the variable $t$ and those  of the form $Ct^{h'}(t)^m$ for some $h'\in \hat{S}$ and $m\geq 0$. 

In the first case, consider the $t^{5m + l}$--term.  Here $\sum_{h \in \hat{S}} \bv{n}(h)i_k(h)=m+\tfrac{l}{5}$, thus the power of $z$ in this term is $5m+1-5m-l$.  Because this is zero, we arrive at $l = 1$, and thus $i_k(h_{\bv{n}}) = \tfrac15$ for all $k$.

The exponent of $z$ in the coefficient of $t^{h'}\cdot(t)^m$ is 
\begin{equation}\label{e:zeropow}
\sum_{k =1}^5 \lfloor \tfrac m5 +i_k(h') \rfloor-m.
\end{equation}

When restricted to $h'\in \hat{S}$, we have 
\[
\sum_{k =1}^5 \lfloor \tfrac m5 +i_k(h') \rfloor-m + \sum_{k=1}^5 \br{\tfrac m5 + i_k(h')}= 1.
\]
Thus expression~\eqref{e:zeropow} is equal to 0 if and only if $\sum_{k=1}^5 \br{\tfrac m5 + i_k(h')}= 1$. One can easily check that this implies $5|m$, therefore $\br{\sum_{h}\bv{n}(h)i_k(h)} =  \br{\tfrac m5 + i_k(h')} = i_k(h')$.
This gives the other terms of $\bv{G}(\bt)$. 
\end{proof}

Now we are prepared to state the mirror theorem.

\begin{theorem}[LG Mirror Theorem]\label{t:mirror}
With $F(\bt)$ and $\bv{G}(\bt)$ as above, and $\bt$ as in \eqref{e:tsmall}, 
\begin{equation}\label{e:mirror}
J^{(W,G)}(\btau(\bt),z)=\frac{I^{(W,G)}(\bt,z)}{F(\bt)} \qquad \text{where } 
 \btau(\bt)=\tfrac{\bv{G}(\bt)}{F(\bt)}.
\end{equation}
\end{theorem}

\begin{proof}
Recall that the $J$--function is uniquely characterized by the fact that is lies on $\sL^{(W,G)}$ and is of the form $z\phi_{\jw}+\bt+\cO(z^{-1})$. The theorem follows from this fact and the previous lemma. 
\end{proof}

\begin{remark}The function $\btau(\bt)$ is referred to as the mirror transformation.
\end{remark}

Let $J_h^{(W,G)}(t,z)$ denote the derivative 
\[
J_h^{(W,G)}(t,z):=z\frac{\partial}{\partial t^h}J^{(W,G)}(\bt,z)|_{\bt=t}.
\] 
Recall by \eqref{e:Jgens} that these functions determine the small cone $\sL^{(W,G)}_{small}$. The rest of the section will be devoted to calculating these functions. 
In fact as we shall see it is sufficient to compute $J_h^{(W,G)}(\bt, z)$ for $\phi_h$ of degree at most two. These will determine all others. 

Expand $I^{(W,G)}(\bt, z)$ in terms of powers of $t^h$ for $h \neq \jw^2$
\[
I^{(W,G)}(\bt,z)=I^{(W,G)}_\jw(t,z)+ \tfrac{1}{z}\Big(\sum_{h}t^hI^{(W,G)}_h(t,z)\Big)+\left(\tfrac{1}{z}
\right)^2\Big(\sum_{h_1,h_2}t^{h_1}t^{h_2}I^{(W,G)}_{h_1,h_2}(t,z)\Big)+\cdots
\]
so that 
\begin{equation}\label{e:Ih}
I^{(W,G)}_h(t,z)=z\frac{\partial}{\partial t^h} I^{(W,G)}(\bt,z)|_{\bt=t}.
\end{equation}

As an immediate consequence of the previous theorem and Lemma~\ref{l:I} we obtain the following corollary.

\begin{corollary}\label{c:mirror}
Given $h\in \hat{S}$ with $\deg_W{\phi_h}\leq 2$, $\phi_h \neq \phi_{\jw^2}$, there exist functions $F_0(t)$, $G_{\jw^2}$, and $G_h(t)$ determined by $I^{(W,G)}_h(t,z)$ such that $F_0$ and $G_h$ are invertible, and 
\[
J^{(W,G)}_h(\tau(t),z)=\frac{I^{(W,G)}_h(t,z)}{G_h(t)} \quad \text{ where } \tau(t)=\frac{G_{\jw^2}(t)}{F_0(t)}
\]
\end{corollary}

\begin{proof}
For $h=\jw$ this follows by setting $\bt=t$. 

For the other $h$ use equation~\eqref{e:mirror}, differentiate both sides with respect to $t^h$, and set $\bt=t$. By equation~\eqref{e:Ih}, the left hand side equals
\[
\frac{G_h(t)}{F_0(t)}J^{(W,G)}_h(\tau(t), z)
\]
and the right hand side equals
\[
\frac{I^{(W,G)}_h(t,z)}{F_0(t)}
\]
as desired. 
\end{proof}

\begin{remark}\label{r:mirror}
To justify the fact that we call Theorem~\ref{t:mirror} and its corollary a ``mirror theorem,'' one can check that up to a factor of $t$ or $t^2$, the functions $I^{(W,G)}_h(t,z)$ satisfy the Picard--Fuchs equations of the mirror family $M_\psi$ around $\psi = 0$.  These equations are found in \cite{LSh}, Table~1 (where the variable $t$ in our paper corresponds to $\psi$ in \cite{LSh}).  One may check this fact directly, or it follows immediately from Theorem~\ref{t:U}.
Combining this fact with Corollary~\ref{c:mirror} yields Theorem~\ref{t:intromirror}.
\end{remark}

The functions $I_h(t,z)$ may be computed directly from \eqref{e:modif}.  We list below $I_h(t,z)$ for $h \in \hat{S} \setminus \jw^2$ satisfying $\deg(\phi_{h}) \leq 2$.  These formulas will be needed in the next section.
\begin{enumerate}
\item[(i)] For $h = \jw$,
\[
tI^{(W,G)}_\jw(t,z) = \sum_{k = 1, 2, 3, 4} \phi_{\jw^{k}} z^{2 - k} \sum_{l\geq 0} t^{k + 5l} \frac{\Gamma((k + 5l)/5)^5}{\Gamma(k/5)\Gamma(k + 5l)}.
\]

\item[(ii)] For $h = (\tfrac 15,\tfrac 15, \tfrac 15,\tfrac 35,\tfrac 45)$ and $h_1 = (\tfrac 45,\tfrac 45,\tfrac 45,\tfrac 15,\tfrac 25)$,
\[
\begin{split}
I^{(W,G)}_h(t,z) 
 	=& z\phi_h  \sum_{l \geq 0} t^{ 5l} \frac{\Gamma((1 + 5l)/5)^3\Gamma((3+ 5l)/5)\Gamma((4 + 5l)/5)}{\Gamma(1/5)^3\Gamma(3/5)\Gamma(4/5)\Gamma(1 + 5l)} \\ 
	+& \phi_{h_1}
 \frac{2}{25} \sum_{l \geq 0} t^{3 + 5l} \frac{\Gamma((4 + 5l)/5)^3\Gamma((6+ 5l)/5)\Gamma((7+ 5l)/5)}{\Gamma(4/5)^3\Gamma(6/5)\Gamma(7/5)\Gamma(4 + 5l)}
\end{split}
\]

\item[(iii)] For $h = (\tfrac 25,\tfrac 25,\tfrac 25,\tfrac 35,\tfrac 15)$ and $h_1 = (\tfrac 35,\tfrac 35,\tfrac 35,\tfrac 45,\tfrac 25)$, 
\[ 
\begin{split}
I^{(W,G)}_h(t,z) = &z\phi_h \sum_{l \geq 0} t^{5l} \frac{\Gamma((2 + 5l)/5)^3\Gamma((3+ 5l)/5)\Gamma((1 + 5l)/5)}{\Gamma(2/5)^3\Gamma(3/5)\Gamma(1/5)\Gamma(1 + 5l)} \\
&+\:\phi_{h_1} \sum_{l \geq 0} t^{1 + 5l} \frac{\Gamma((3 + 5l)/5)^3\Gamma((4+ 5l)/5)\Gamma((2+ 5l)/5)}{\Gamma(3/5)^3\Gamma(4/5)\Gamma(2/5)\Gamma(2 + 5l)}.
\end{split}
\]

\item[(iv)] For $h = (\tfrac 15,\tfrac 15,\tfrac 25,\tfrac 25,\tfrac 45)$ and $h_1 = (\tfrac 35,\tfrac 35,\tfrac 45,\tfrac 45,\tfrac 15)$, 
\[ 
\begin{split}
I^{(W,G)}_h(t,z) 
	= &z\phi_h \sum_{l \geq 0} t^{ 5l} \frac{\Gamma((1 + 5l)/5)^2\Gamma((2+ 5l)/5)^2\Gamma((4 + 5l)/5)}{\Gamma(1/5)^2\Gamma(2/5)^2\Gamma(4/5)\Gamma(1 + 5l)}  \\
	+&\frac{\phi_{h_1}}{5} \sum_{l \geq 0} t^{2 + 5l} \frac{\Gamma((3 + 5l)/5)^2\Gamma((4+ 5l)/5)^2\Gamma((6 + 5l)/5)}{\Gamma(3/5)^2\Gamma(4/5)^2\Gamma(6/5)\Gamma(3 + 5l)}.
\end{split}
\]

\item[(v)] For $h = (\tfrac 15,\tfrac 15,\tfrac 35,\tfrac 35,\tfrac 25)$ and $h_1 = (\tfrac 25,\tfrac 25,\tfrac 45,\tfrac 45,\tfrac 35)$,
\[ 
\begin{split}
I^{(W,G)}_h(t,z) 
	= &z\phi_h \sum_{l \geq 0} t^{ 5l} \frac{\Gamma((1 + 5l)/5)^2\Gamma((3+ 5l)/5)^2\Gamma((2 + 5l)/5)}{\Gamma(1/5)^2\Gamma(3/5)^2\Gamma(2/5)\Gamma(1 + 5l)} \\
	+\:&\phi_{h_1} \sum_{l \geq 0} t^{1 + 5l} \frac{\Gamma((2 + 5l)/5)^2\Gamma((4+ 5l)/5)^2\Gamma((3 + 5l)/5)}{\Gamma(2/5)^2\Gamma(4/5)^2\Gamma(3/5)\Gamma(2 + 5l)}.
\end{split}
\]
\end{enumerate}

%% file: thecorrespondence.tex
\section{The LG/CY correspondence for the mirror quintic}

\subsection{The state space correspondence}\label{s:ssc}
An isomorphism between the Landau--Ginzberg state space and the cohomology of corresponding Calabi--Yau hypersurfaces is proven in \cite{ChR2}.  In the case of the mirror quintic, the work implies in particular an isomorphism between $H^{even}_{CR}(\cW)$ and  $\sH_{W,G}^{nar}$ as graded vector spaces.  We will describe the correspondence explicitly below.  Recall that $H^{even}_{CR}(\cW)$ can be split into summands indexed by $g \in \tilde{S}$, where $\tilde{S}$ is composed of elements $g = (r_1, r_2, r_3, r_4, r_5) \in G$ such that at least two $r_i$ are 0.  The basis for $\sH_{W,G}^{nar}$ on the other hand is given by $\{\phi_h\}_{h \in \hat{S}}$ where $\hat{S}$ runs over elements $h=(r_1, r_2, r_3, r_4, r_5) \in G$ such that $r_i  \neq 0$ for all $i$.  

\subsubsection{$\dim(\cW_g) = 3$}

For $g=e$, map

\begin{align*}
\mu: H^i \mapsto \phi_{\jw^{i+1}}.
\end{align*}

\subsubsection{$\dim(\cW_g) = 1$}

For $g=(0,0,0,\tfrac 25,\tfrac 35)$, let $h = (\tfrac 15,\tfrac 15,\tfrac 15,\tfrac 35,\tfrac 45)$ and $h_1 = (\tfrac 45,\tfrac 45,\tfrac 45,\tfrac 15,\tfrac 25)$, then
\begin{align*}
\mu: \ii_g &\mapsto \phi_h\\
 \ii_{g}H & \mapsto \phi_{h_1}.
\end{align*}
For $g=(0,0,0,\tfrac 15,\tfrac 45)$, let $h = (\tfrac 25,\tfrac 25,\tfrac 25,\tfrac 35,\tfrac 15)$ and $h_1 = (\tfrac 35,\tfrac 35,\tfrac 35,\tfrac 45,\tfrac 25)$, then
\begin{align*}
\mu: \ii_g &\mapsto \phi_{h}\\
 \ii_{g}H & \mapsto \phi_{h_1}.
\end{align*}

\subsubsection{$\dim(\cW_g) = 0$}

Let \begin{align*}
\mu: \ii_g &\mapsto \phi_h,
\end{align*} where, 

if $g=(0,0,\tfrac 15,\tfrac 15,\tfrac 35)$, $h = (\tfrac 15,\tfrac 15,\tfrac 25,\tfrac 25,\tfrac 45)$; 

if $g = (0,0,\tfrac 45,\tfrac 45, \tfrac 25)$,  $h = (\tfrac 45,\tfrac 45,\tfrac 35,\tfrac 35,\tfrac 15)$;

if $g=(0,0,\tfrac 25,\tfrac 25,\tfrac 15)$, $h = (\tfrac 15,\tfrac 15,\tfrac 35,\tfrac 35,\tfrac 25)$;

if $g = (0,0,\tfrac 35,\tfrac 35,\tfrac 45)$, $h =  (\tfrac 45,\tfrac 45,\tfrac 25,\tfrac 25,\tfrac 35)$.
\\
\\
If $g$ is a permutation of one of the above, define the map by permuting the $h$ elements accordingly.  By extending the above identification linearly, we obtain a map \[\mu: H^{even}_{CR}(\cW) \to \sH^{nar}_{W,G}\] identifying the state spaces.
Note that this identification preserves the grading and (up to a constant factor) preserves the pairing.  

\subsection{Analytic continuation of $I^{\cW}$}

Let $J^{\cW}(\bs, z)$ denote the (big) $J$--function of the mirror quintic $\cW$.  Let $s^g$ denote the dual coordinate to the fundamental class $\ii_g$ on $\cW_g$.  We define 
\[
J^{\cW}_g(s, z) := z\frac{\partial}{\partial s^g} J^{\cW}(\bs, z) |_{\bs = sH}.
\]
For $g$ of age at most 1, we know by \cite{LSh} that 
\begin{equation}\label{e:A-model}
 J_g^{\cW}(\sigma(s),z) = 
 \frac{I^{\cW}_g(s,z)}{H_g(s)} \qquad \text{where } 
 \sigma(s) = \frac{G_0(s)}{F_0(s)}, 
\end{equation}
where here $H_g, G_0,$ and $F_0$ are explicitly determined functions, and $I^{\cW}_g$ is given below.  Let $q = e^s$, then
\begin{enumerate}
\item[(i)]If $g = e = (0,0,0,0,0)$, 
\begin{equation*} 
 \begin{split}
&I^{\cW}_e(q,z) = zq^{ H/z}\left(1 + \sum_{\langle d \rangle = 0}  q^{d} 
\frac{\prod \limits_{1 \leq m\leq5d}(5H + mz)}{\prod \limits_{\substack{0 < b \leq d \\ 
\langle b \rangle = 0}} (H + bz)^5}\right).
 \end{split}
\end{equation*}  

\item[(ii)] If $g = (0,0,0,r_1, r_2)$, 

\begin{equation*}
 \begin{split}
 &I^{\cW}_g(q,z) =  \\ 
 zq^{H/z}\ii_g &\Bigg(1 +  \sum_{\langle d \rangle = 0}   q^{d} 
 \frac{\prod \limits_{1 \leq m\leq5d}(5H + mz)}{\prod 
 \limits_{\substack{0 < b \leq d \\ \langle b \rangle = 0}} (H + bz)^3\prod 
 \limits_{\substack{0 < b \leq d \\ \langle b \rangle = r_2}} 
 (H + bz)\prod 
 \limits_{\substack{0 < b \leq d \\ \langle b \rangle = r_1}} 
  (H + bz)}\Bigg).
 \end{split}
\end{equation*}
\item[(iii)] If $g = (0,0,r_1, r_1, r_2)$, let 
$g_1 = (\langle -r_1\rangle, \langle-r_1\rangle, 0,0, \langle r_2 - r_1\rangle)$.  Then
 
\begin{equation*}
 \begin{split} 
  &I^{\cW}_g(q,z) = \\
  zq^{ H/z}\ii_g &\left(1 + \sum_{\langle d \rangle = 0}  q^{d} \frac{\prod 
  \limits_{1 \leq m\leq5d}(5H+ mz)}{\prod 
    \limits_{ \substack{0 < b \leq d \\ \langle b \rangle = 0}} (H+bz)^2\prod 
   \limits_{\substack{0 < b \leq d \\ \langle b \rangle = \langle 3r_2 \rangle}} 
  (H+bz)^2\prod 
 \limits_{\substack{0 < b \leq d \\ \langle b \rangle = \langle 2r_1\rangle}}
 (H+bz)} \right) \\  
 +\:zq^{ H/z}\ii_{g_1} &\left(
 \sum_{\langle d \rangle = r_1}  q^{d} 
 \frac{\prod \limits_{1 \leq m\leq5d}(5H+ mz)}{\prod 
 \limits_{\substack{0 < b \leq d \\ \langle b \rangle = r_1}}
 (H+bz)^2\prod 
 \limits_{\substack{0 < b \leq d \\ \langle b \rangle = 0}} (H+bz)^2\prod 
 \limits_{\substack{0 < b \leq d \\ \langle b \rangle = r_2}}
 (H+bz)}\right) 
 \end{split}
\end{equation*}
\end{enumerate}

We will analytic continue each of the above $I$--functions from $q = 0$ to $t = q^{-1/5} = 0$ using the Mellon--Barnes method as in \cite{ChR}.  

\subsubsection{$g = e = (0,0,0,0,0)$}  The $I$--function $I_e^{\cW}$ is identical to the $I$--function in \cite[Equation (47)]{ChR}, after reinterpreting $H$ as the hyperplane class in $H^2(\cW)$.  We recall their analytic continuation and symplectic transformation in \ref{symp:dim3}.

\subsubsection{$g = (0,0,0,r_1, r_2)$}

The Gamma function satisfies 
\[ \Gamma(z+n)/\Gamma(z) = (z)(z+1) \cdots (z+n-1) \] and consequently
\[ \Gamma(1 + x/z + l)/\Gamma(1+x/z) = z^{-l} \prod_{k=1}^l(x+kz).\]

With this we can rewrite our $I$--functions.  In the present case we obtain

\[
\begin{split}
 &I^{\cW}_g(q,z)  =  z\ii_gq^{ H/z}  \cdot \\
 &\sum_{\langle{d} \rangle = 0} q^d
 \frac{\Gamma(1 + 5H/z + 5d) \Gamma(1 + H/z)^3 \Gamma( r_1 + H/z)\Gamma(r_2 + H/z)}{\Gamma(1 + 5H/z ) \Gamma(1 + H/z + d)^3 \Gamma( r_1 + H/z + d)\Gamma(r_2 + H/z + d)}.
 \end{split}
\]
The function $1/(e^{2 \pi i s} - 1)$ has simple poles at each integer with residue 1.  From this we can rewrite the function as a contour integral
\[
\begin{split}
I^{\cW}_g(q,z)  = &z\ii_g  q^{ H/z} \frac{\Gamma(1 + H/z)^3 \Gamma( r_1 + H/z)\Gamma(r_2 + H/z)}{\Gamma(1 + 5H/z )} \cdot \\
&\int_C \frac{1}{e^{2 \pi i s} - 1} q^s
 \frac{\Gamma(1 + 5H/z + 5s) }{ \Gamma(1 + H/z + s)^3 \Gamma( r_1 + H/z + s)\Gamma(r_2 + H/z + s)}.
 \end{split}
\]
where the curve $C$ goes from $+ i \infty$ to $- i \infty$ and encloses all nonnegative integers to the right. 

By closing the curve to the left, we obtain an expansion in terms of $t = q^{-1/5}$.  The Gamma function has poles at nonpositive integers, so we obtain a sum of residues at $s = -1 - l$ for $l \geq 0$ and $s = -H/z - m/5$ for $m \geq 1$.  In this case, at negative integers, the residue is a multiple of $H^2$, and so vanishes on $\cW_g$.  The residue similarly vanishes at $s = -H/z - m/5$ when $m$ is congruent to $0$, $5r_1$, or $5r_2$.  For the remaining values of $m$, we use
\[
\Res_{s = -H/z - m/5} \Gamma(1 + 5H/z + 5s) = -\frac{1}{5} \frac{(-1)^m}{\Gamma(m)},
\]
to obtain
\[
\begin{split}
& I^{\cW'}_g(t,z)  = z\ii_g   \frac{\Gamma(1 + H/z)^3 \Gamma( r_1 + H/z)\Gamma(r_2 + H/z)}{5 \Gamma(1 + 5H/z )} \cdot \\ 
&\sum_{\substack{0<m \\ m\not \equiv 0, 5r_1, 5r_2}} \frac{(-\xi)^m 2\pi i  }{e^{-2 \pi i H/z} - \xi^m} \frac{t^m}{\Gamma(m) \Gamma(1 - m/5)^3 \Gamma(r_1 - m/5)\Gamma(r_2 - m/5)}.
\end{split}
\]

Here the prefactor of $q^{H/z}$ cancels with a term in each residue.  Note that $\Gamma(r_1 - m/5) = \Gamma(1 - r_2 - m/5)$.  Recalling the identity $\Gamma(x)\Gamma(1-x) = \pi / \sin(\pi x)$, we simplify the above expression as
\[
\begin{split}
I^{\cW'}_g(t,z)  &= z\ii_g   \frac{\Gamma(1 + H/z)^3 \Gamma( r_1 + H/z)\Gamma(r_2 + H/z)}{5 \Gamma(1 + 5H/z )} \cdot \\ 
\sum_{\substack{0<m\\ m \not \equiv 0, 5r_1, 5r_2}} &\frac{(-\xi)^m 2\pi i  }{e^{-2 \pi i H/z} - \xi^m} \frac{t^m \pi^{-5}\Gamma(m/5)^3  \Gamma(r_1 + m/5)  \Gamma(r_2 + m/5)}{\Gamma(m) (\sin(\pi m/5))^{-3} (\sin(\pi(r_1 + m/5)))^{-1}(\sin(\pi(r_2 + m/5)))^{-1}} \\
&=z\ii_g   \frac{\Gamma(1 + H/z)^3 \Gamma( r_1 + H/z)\Gamma(r_2 + H/z)}{5 \Gamma(1 + 5H/z )} 
\cdot \\
\sum_{\substack{0<k<5 \\ k \not \equiv 0, 5r_1, 5r_2}}& \left( \frac{(-\xi)^k 2\pi i }{e^{-2 \pi i H/z} - \xi^k} \frac{1}{\textstyle  \Gamma(1-k/5)^3  \Gamma(1-(r_1 + k/5))\Gamma(1-(r_2 + k/5))} \right. 
\cdot \\
& \left. \sum_{l \geq 0} t^{k + 5l} \frac{\Gamma((k + 5l)/5)^3\Gamma(r_1 + (k+ 5l)/5)\Gamma(r_2 + (k + 5l)/5)}{\Gamma(k/5)^3\Gamma(r_1 + k/5)\Gamma(r_2 + k/5)\Gamma(k + 5l)}\right).
\end{split}
\]

\subsubsection{$g = (0,0,r_1, r_1, r_2)$} Let $g_1 = (\langle -r_1\rangle, \langle-r_1\rangle, 0,0, \langle r_2 - r_1\rangle)$.
Re--writing $I^{\cW}_g(t,z)$ in terms of Gamma functions yields 
\[
 \begin{split} 
  I^{\cW}_g(q,z) &= \\
  z &\ii_g \Gamma(1 - r_1)^2 \Gamma(1 - r_2)\left(\sum_{\langle d \rangle = 0}  q^{d} \frac{\Gamma(1 + 5d)}{\Gamma(1 + d)^2 \Gamma(1 - r_1 + d)^2 \Gamma(1 - r_2 + d)} \right) \\  
 +\: &\ii_{g_1}  
 \Gamma(r_1)^2 \Gamma(r_2)\left(\sum_{\langle d \rangle = r_1}  q^{d} \frac{\Gamma(1 + 5d)}{\Gamma(1 + d)^2 \Gamma(1 - r_1 + d)^2 \Gamma(1 - r_2 + d)} \right)\\
  \: =z\ii_g &\Gamma(1 - r_1)^2 \Gamma(1 - r_2)\int_C \frac{1}{e^{2 \pi i s} - 1} q^s
\frac{\Gamma(1 + 5s)}{\Gamma(1 + s)^2 \Gamma(1 - r_1 + s)^2 \Gamma(1 - r_2 + s)}
 \\
 +\: &\ii_{g_1}  \Gamma(r_1)^2 \Gamma(r_2)
  \int_C \frac{1}{e^{2 \pi i (s - r_1)} - 1} q^s
\frac{\Gamma(1 + 5s)}{\Gamma(1 + s)^2 \Gamma(1 - r_1 + s)^2 \Gamma(1 - r_2 + s)}
\end{split}
\]
Analytic continuing along the other side using the same method as above we obtain
\[
 \begin{split} 
  I^{\cW'}_g(t,z) 
=&z\frac{\ii_g}{5}   \Gamma(1 - r_1)^2 \Gamma(1 - r_2)
\cdot \\
&\sum_{0<k<5 | k  \equiv 5r_1, 5r_2} \left( \frac{(-\xi)^k 2\pi i }{1 - \xi^k} \frac{1}{\textstyle  \Gamma(1-k/5)^2 \Gamma(1 - (r_1 + k/5))^2 \Gamma(1-(r_2 + k/5))} \right. 
\cdot \\
 &\left. \sum_{l \geq 0} t^{k + 5l} \frac{\Gamma((k + 5l)/5)^2\Gamma(r_1 + (k+ 5l)/5)^2\Gamma(r_2 + (k + 5l)/5)}{\Gamma(k/5)^2\Gamma(r_1 + k/5)^2\Gamma(r_2 + k/5)\Gamma(k + 5l)}\right) \\
 + \: &\ii_{g_1}  \Gamma(r_1)^2 \Gamma(r_2)\cdot \\
 &\sum_{0<k<5 | k  \equiv 5r_1, 5r_2} \left( \frac{(-1)^k\xi^{k+ 5r_1} 2\pi i }{1 - \xi^{k+5r_1}} \frac{1}{\textstyle  \Gamma(1-k/5)^2 \Gamma(1 - (r_1 + k/5))^2 \Gamma(1-(r_2 + k/5))} \right. 
\cdot \\
 &\left. \sum_{l \geq 0} t^{k + 5l} \frac{\Gamma((k + 5l)/5)^2\Gamma(r_1 + (k+ 5l)/5)^2\Gamma(r_2 + (k + 5l)/5)}{\Gamma(k/5)^2\Gamma(r_1 + k/5)^2\Gamma(r_2 + k/5)\Gamma(k + 5l)}\right).
 \end{split}
\]

\subsection{The symplectic transformation}\label{ss:st}

\subsubsection{$g = e$}\label{symp:dim3}
Here we recall calculations from \cite{ChR}, and the symplectic transformation which they compute.  Analytic continuation of $I^{\cW}_e(t,z)$ yields
\[
I^{\cW '}_e(t,z) = z \frac{\Gamma(1 + H/z)^5}{5\Gamma(1 + 5H/z)} \sum_{k = 1, 2, 3, 4} \frac{(-\xi)^k 2 \pi i }{e^{-2\pi i H/z} - \xi^k}\frac{1}{\Gamma(1 - k/5)^5} \sum_{l\geq 0} t^{k + 5l} \frac{\Gamma((k + 5l)/5)^5}{\Gamma(k/5)\Gamma(k + 5l)}.
\]
On the other hand 
\[
tI^{(W,G)}_\jw(t,z) = \sum_{k = 1, 2, 3, 4} \phi_{\jw^{k}} z^{2 - k} \sum_{l\geq 0} t^{k + 5l} \frac{\Gamma((k + 5l)/5)^5}{\Gamma(k/5)\Gamma(k + 5l)}.
\]
Thus the transformation 
\[\UU_{\jw^{k}}: \phi_{\jw^{k}} \mapsto z^{k-1} \frac{\Gamma(1 + H/z)^5}{\Gamma(1 + 5H/z)} \frac{(-\xi)^k 2 \pi i }{e^{-2\pi i H/z} - \xi^k}\frac{1}{\Gamma(1 - k/5)^5}\]
sends $\frac{t}{5}I^{(W,G)}_{\jw}(t,z)$ to $I^{\cW '}_e(t, z)$.  

\subsubsection{$g = (0,0,0,\tfrac 25,\tfrac 35)$} In this case
\[
\begin{split}
I^{\cW'}_g(t,z)  = z\ii_g &   \frac{\Gamma(1 + H/z)^3 \Gamma( 2/5 + H/z)\Gamma(3/5 + H/z)}{5 \Gamma(1 + 5H/z )} 
\cdot \\
	 \Bigg( & \frac{(-\xi) 2\pi i }{e^{-2 \pi i H/z} - \xi} \frac{1}{\textstyle  \Gamma(1-1/5)^3  \Gamma(1-3/5)\Gamma(1-4/5)} \cdot \\
	& \sum_{l \geq 0} t^{1 + 5l} \frac{\Gamma((1 + 5l)/5)^3\Gamma((3+ 5l)/5)\Gamma((4 + 5l)/5)}{\Gamma(1/5)^3\Gamma(3/5)\Gamma(4/5)\Gamma(1 + 5l)}\\
	& +\frac{(-\xi)^4 2\pi i }{e^{-2 \pi i H/z} - \xi^4} \frac{1}{\textstyle  \Gamma(1-4/5)^3  \Gamma(1-6/5)\Gamma(1-7/5)} \cdot \\
	& \sum_{l \geq 0} t^{4 + 5l} \frac{\Gamma((4 + 5l)/5)^3\Gamma((6+ 5l)/5)\Gamma((7+ 5l)/5)}{\Gamma(4/5)^3\Gamma(6/5)\Gamma(7/5)\Gamma(4 + 5l)}\Bigg).
\end{split}
\]
Using the relation $\Gamma(1+x)=x\Gamma(x)$, we can rewrite the last summand, which gives us
\[
\begin{split}
I^{\cW'}_g(t,z)  = z\ii_g &   \frac{\Gamma(1 + H/z)^3 \Gamma( 2/5 + H/z)\Gamma(3/5 + H/z)}{5 \Gamma(1 + 5H/z )} 
\cdot \\
 \Bigg( &\frac{(-\xi) 2\pi i }{e^{-2 \pi i H/z} - \xi} \frac{1}{\textstyle  \Gamma(1-1/5)^3  \Gamma(1-3/5)\Gamma(1-4/5)}  
\cdot \\
& \sum_{l \geq 0} t^{1 + 5l} \frac{\Gamma((1 + 5l)/5)^3\Gamma((3+ 5l)/5)\Gamma((4 + 5l)/5)}{\Gamma(1/5)^3\Gamma(3/5)\Gamma(4/5)\Gamma(1 + 5l)}\\
& + \frac{(-\xi)^4 2\pi i }{e^{-2 \pi i H/z} - \xi^4} \frac{1}{\textstyle  \Gamma(1-4/5)^3  \Gamma(1-1/5)\Gamma(1-2/5)}  
\cdot \\
& \frac{2}{25} \sum_{l \geq 0} t^{4 + 5l} \frac{\Gamma((4 + 5l)/5)^3\Gamma((6+ 5l)/5)\Gamma((7+ 5l)/5)}{\Gamma(4/5)^3\Gamma(6/5)\Gamma(7/5)\Gamma(4 + 5l)}\Bigg).
\end{split}
\]
If $h = (\tfrac 15,\tfrac 15, \tfrac 15,\tfrac 35,\tfrac 45)$ and $h_1 = (\tfrac 45,\tfrac 45,\tfrac 45,\tfrac 15, \tfrac25 )$, we see that the transformation 
\begin{align*}
&\begin{split}
\UU_h: \phi_h \mapsto \ii_g   &\frac{\Gamma(1 + H/z)^3 \Gamma( 2/5 + H/z)\Gamma(3/5 + H/z)}{ \Gamma(1 + 5H/z )} \cdot \\
& \frac{(-\xi) 2\pi i }{e^{-2 \pi i H/z} - \xi} \frac{1}{\textstyle  \Gamma(1-1/5)^3  \Gamma(1-3/5)\Gamma(1-4/5)},
\end{split}\\
&\begin{split}
\UU_{h_1}: \phi_{h_1} \mapsto  z \ii_g &\frac{\Gamma(1 + H/z)^3 \Gamma( 2/5 + H/z)\Gamma(3/5 + H/z)}{ \Gamma(1 + 5H/z )} \cdot \\ 
&  \frac{(-\xi)^4 2\pi i }{e^{-2 \pi i H/z} - \xi^4} \frac{1}{\textstyle  \Gamma(1-4/5)^3  \Gamma(1-1/5)\Gamma(1-2/5)} \end{split}
\end{align*}
sends $\frac{t}{5}I^{(W,G)}_h(t, z)$ to $I^{\cW '}_g(t,z)$.

\subsubsection{$g = (0,0,0,\tfrac 15,\tfrac 45)$} 
\[
\begin{split}
I^{\cW'}_g(t,z)  
 = z\ii_g &   \frac{\Gamma(1 + H/z)^3 \Gamma( 1/5 + H/z)\Gamma(4/5 + H/z)}{5 \Gamma(1 + 5H/z )}\cdot \\
	\Bigg( & \frac{(-\xi)^2 2\pi i }{e^{-2 \pi i H/z} - \xi^2} \frac{1}{\textstyle  \Gamma(1-2/5)^3  \Gamma(1-3/5)\Gamma(1-1/5)}\\
	& \left(-\tfrac{1}{5}\right)\sum_{l \geq 0} t^{2 + 5l} \frac{\Gamma((2 + 5l)/5)^3\Gamma((3+ 5l)/5)\Gamma((1 + 5l)/5)}{\Gamma(2/5)^3\Gamma(3/5)\Gamma(1/5)\Gamma(1 + 5l)}\\
	&+\frac{(-\xi)^3 2\pi i }{e^{-2 \pi i H/z} - \xi^3} \frac{1}{\textstyle  \Gamma(1-3/5)^3  \Gamma(1-4/5)\Gamma(1-2/5)}\cdot \\
	& \left(-\tfrac{1}{5}\right)\sum_{l \geq 0} t^{3 + 5l} \frac{\Gamma((3 + 5l)/5)^3\Gamma((4+ 5l)/5)\Gamma((2+ 5l)/5)}{\Gamma(3/5)^3\Gamma(4/5)\Gamma(2/5)\Gamma(2 + 5l)}\Bigg).
\end{split}
\]
If $h = (\tfrac 25,\tfrac 25,\tfrac 25,\tfrac 35,\tfrac 15)$ and $h_1 = (\tfrac 35,\tfrac 35,\tfrac 35,\tfrac 45,\tfrac 25)$, the transformation
\[
\begin{split}
\UU_h: \phi_h \mapsto \ii_g   &  \frac{\Gamma(1 + H/z)^3 \Gamma( 1/5 + H/z)\Gamma(4/5 + H/z)}{ \Gamma(1 + 5H/z )} 
\cdot \\
& \frac{-(\xi)^2 2\pi i }{e^{-2 \pi i H/z} - \xi^2} \frac{1}{\textstyle  \Gamma(1-2/5)^3  \Gamma(1-3/5)\Gamma(1-1/5)},\end{split}
\]
\[
\begin{split}
\UU_{h_1}: \phi_{h_1} \mapsto  z \ii_g & \frac{\Gamma(1 + H/z)^3 \Gamma( 1/5 + H/z)\Gamma(4/5 + H/z)}{ \Gamma(1 + 5H/z )} 
\cdot  \\ 
&  \frac{(\xi)^3 2\pi i }{e^{-2 \pi i H/z} - \xi^3} \frac{1}{\textstyle  \Gamma(1-3/5)^3  \Gamma(1-4/5)\Gamma(1-2/5)} \end{split}
\] 
sends $\frac{t^2}{25}I^{(W,G)}_h(t, z)$ to $I^{\cW '}_g(t,z)$.

\subsubsection{$g = (0,0,\tfrac 15,\tfrac 15,\tfrac 35)$}  Letting $g_1 = (\tfrac45, \tfrac45, 0,0, \tfrac25)$,
\[
 \begin{split} 
  I^{\cW'}_g(t,z) =z\frac{\ii_g}{5}&  \Gamma(1 - 1/5)^2 \Gamma(1 - 3/5)
\cdot \\
	& \Bigg( \frac{(-\xi) 2\pi i }{1 - \xi} \frac{1}{\textstyle  \Gamma(1-1/5)^2 \Gamma(1 - 2/5)^2 \Gamma(1-4/5)} \cdot \\
	& \sum_{l \geq 0} t^{1 + 5l} \frac{\Gamma((1 + 5l)/5)^2\Gamma((2+ 5l)/5)^2\Gamma((4 + 5l)/5)}{\Gamma(1/5)^2\Gamma(2/5)^2\Gamma(4/5)\Gamma(1 + 5l)}\\
	&+ \frac{(-\xi)^3 2\pi i }{1 - \xi^3} \frac{1}{\textstyle  \Gamma(1-3/5)^2 \Gamma(1 - 4/5)^2 \Gamma(1-1/5)} \cdot \\
	& \left(-\tfrac{1}{5}\right) \sum_{l \geq 0} t^{3 + 5l} \frac{\Gamma((3 + 5l)/5)^2\Gamma((4+ 5l)/5)^2\Gamma((6 + 5l)/5)}{\Gamma(3/5)^2\Gamma(4/5)^2\Gamma(6/5)\Gamma(3 + 5l)}\Bigg) \\
 	+ \frac{\ii_{g_1}}{5} &\Gamma(1/5)^2 \Gamma(3/5)\cdot \\
	& \Bigg( \frac{(-1)\xi^{2} 2\pi i }{1 - \xi^{2}} \frac{1}{\textstyle  \Gamma(1-1/5)^2 \Gamma(1 - 2/5)^2 \Gamma(1-4/5)} \cdot \\
	& \sum_{l \geq 0} t^{1 + 5l} \frac{\Gamma((1 + 5l)/5)^2\Gamma((2+ 5l)/5)^2\Gamma((4 + 5l)/5)}{\Gamma(1/5)^2\Gamma(2/5)^2\Gamma(4/5)\Gamma(1 + 5l)} \\
	& + \frac{(-1)^3\xi^{4} 2\pi i }{1 - \xi^{4}} \frac{1}{\textstyle  \Gamma(1-3/5)^2 \Gamma(1 - 4/5)^2 \Gamma(1-1/5)} \cdot \\
	& \left(-\tfrac{1}{5}\right) \sum_{l \geq 0} t^{3 + 5l} \frac{\Gamma((3 + 5l)/5)^2\Gamma((4+ 5l)/5)^2\Gamma((6 + 5l)/5)}{\Gamma(3/5)^2\Gamma(4/5)^2\Gamma(6/5)\Gamma(3 + 5l)}\Bigg).
 \end{split}
\]
Letting $h = (\tfrac 15,\tfrac 15,\tfrac 25,\tfrac 25,\tfrac 45)$ and $h_1 = (\tfrac 35,\tfrac 35,\tfrac 45,\tfrac 45,\tfrac 15)$,  
 the transformation
 \[
 \begin{split}
\UU_h: \phi_h \mapsto  \ii_g & \Gamma(1 - 1/5)^2 \Gamma(1 - 3/5)
\cdot \\
	&  \frac{(-\xi) 2\pi i }{1 - \xi} \frac{1}{\textstyle  \Gamma(1-1/5)^2 \Gamma(1 - 2/5)^2 \Gamma(1-4/5)}\\
	+   \ii_{g_1} &  \frac{\Gamma(1/5)^2 \Gamma(3/5)}{ z}\cdot \\
	& \frac{(-1)\xi^{2} 2\pi i }{1 - \xi^{2}} \frac{1}{\textstyle  \Gamma(1-1/5)^2 \Gamma(1 - 2/5)^2 \Gamma(1-4/5)},\end{split}
\]
\[
\begin{split}
\UU_{h_1}: \phi_{h_1} \mapsto  z \ii_g & \Gamma(1 - 1/5)^2 \Gamma(1 - 3/5) \cdot \\
	&  \frac{(\xi)^3 2\pi i }{1 - \xi^3} \frac{1}{\textstyle  \Gamma(1-3/5)^2 \Gamma(1 - 4/5)^2 \Gamma(1-1/5)}\\
	+  \ii_{g_1} & \Gamma(1/5)^2 \Gamma(3/5)\cdot \\
	& \frac{\xi^{4} 2\pi i }{1 - \xi^{4}} \frac{1}{\textstyle  \Gamma(1-3/5)^2 \Gamma(1 - 4/5)^2 \Gamma(1-1/5)},
\end{split}
\]
sends $\frac{t}{5}I^{(W,G)}_h(t, z)$ to $I^{\cW '}_g(t,z)$.

\subsubsection{$g = (0,0,\tfrac 25,\tfrac 25,\tfrac 15)$} Letting $g_1 = (\tfrac35, \tfrac35, 0,0, \tfrac45)$,
\[
\begin{split} 
I^{\cW'}_g(t,z)
	=z\frac{\ii_g}{5} &  \Gamma(1 - 2/5)^2 \Gamma(1 - 1/5) \cdot \\
	\Bigg( & 
%k=2	
	\frac{(-\xi)^2 2\pi i }{1 - \xi^2} \frac{1}{\textstyle  \Gamma(1-2/5)^2 \Gamma(1 - 4/5)^2 \Gamma(1-3/5)} \cdot \\
	& \sum_{l \geq 0} t^{2 + 5l} \frac{\Gamma((2 + 5l)/5)^2\Gamma((4+ 5l)/5)^2\Gamma((3 + 5l)/5)}{\Gamma(2/5)^2\Gamma(4/5)^2\Gamma(3/5)\Gamma(2 + 5l)} \\
%k=1
	& \frac{(-\xi) 2\pi i }{1 - \xi} \frac{1}{\textstyle  \Gamma(1-1/5)^2 \Gamma(1 - 3/5)^2 \Gamma(1-2/5)} \cdot \\
% k=1
	& \sum_{l \geq 0} t^{1 + 5l} \frac{\Gamma((1 + 5l)/5)^2\Gamma((3+ 5l)/5)^2\Gamma((2 + 5l)/5)}{\Gamma(1/5)^2\Gamma(3/5)^2\Gamma(2/5)\Gamma(1 + 5l)}\Bigg) \\
	+ \frac{\ii_{g_1}}{5} &\Gamma(2/5)^2 \Gamma(1/5)\cdot \\
 %k=2
	& \Bigg( \frac{(-1)^2\xi^{4} 2\pi i }{1 - \xi^{4}} \frac{1}{\textstyle  \Gamma(1-2/5)^2 \Gamma(1 - 4/5)^2 \Gamma(1-3/5)} \cdot \\
	& \sum_{l \geq 0} t^{2 + 5l} \frac{\Gamma((2 + 5l)/5)^2\Gamma((4+ 5l)/5)^2\Gamma((3 + 5l)/5)}{\Gamma(2/5)^2\Gamma(4/5)^2\Gamma(3/5)\Gamma(2 + 5l)} \\
%k=1
	& +\frac{(-1)\xi^{3} 2\pi i }{1 - \xi^{3}} \frac{1}{\textstyle  \Gamma(1-1/5)^2 \Gamma(1 - 3/5)^2 \Gamma(1-2/5)} \cdot \\
% k=1 
	&\sum_{l \geq 0} t^{1 + 5l} \frac{\Gamma((1 + 5l)/5)^2\Gamma((3+ 5l)/5)^2\Gamma((2 + 5l)/5)}{\Gamma(1/5)^2\Gamma(3/5)^2\Gamma(2/5)\Gamma(1 + 5l)}\Bigg).
 \end{split}
\]

Letting $h = (\tfrac 15,\tfrac 15,\tfrac 35,\tfrac 35,\tfrac 25)$ and $h_1 = (\tfrac 25,\tfrac 25,\tfrac 45,\tfrac 45,\tfrac 35)$, 
 the transformation
\[
\begin{split}
\UU_h: \phi_h \mapsto & \ii_g     \Gamma(1 - 2/5)^2 \Gamma(1 - 1/5)
\cdot \\
	&  \frac{(-\xi) 2\pi i }{1 - \xi} \frac{1}{\textstyle  \Gamma(1-1/5)^2 \Gamma(1 - 3/5)^2 \Gamma(1-2/5)} \\
	+ &  \ii_{g_1}  \frac{\Gamma(2/5)^2 \Gamma(1/5)}{ z}\cdot \\
	& \frac{(-1)\xi^{3} 2\pi i }{1 - \xi^{3}} \frac{1}{\textstyle  \Gamma(1-1/5)^2 \Gamma(1 - 3/5)^2 \Gamma(1-2/5)} ,
\end{split}
\]
\[
\begin{split}
\UU_{h_1}: \phi_{h_1} \mapsto  &z \ii_g  \Gamma(1 - 2/5)^2 \Gamma(1 - 1/5) \cdot \\
	&  \frac{(-\xi)^2 2\pi i }{1 - \xi^2} \frac{1}{\textstyle  \Gamma(1-2/5)^2 \Gamma(1 - 4/5)^2 \Gamma(1-3/5)} \\
	+&  \ii_{g_1} \Gamma(2/5)^2 \Gamma(1/5)\cdot \\
	& \frac{(-1)^2\xi^{4} 2\pi i }{1 - \xi^{4}} \frac{1}{\textstyle  \Gamma(1-2/5)^2 \Gamma(1 - 4/5)^2 \Gamma(1-3/5)},
	\end{split}
\]
sends $\frac{t}{5}I^{(W,G)}_h(t, z)$ to $I^{\cW '}_g(t,z)$.

\subsubsection{Putting things together}

The above calculations define a map \[\UU_h : \phi_h \to \sV^{\cW}\] for each $h \in \hat{S}$.
Extending linearly, we may define the transformation $\UU$,
\[
\UU := \bigoplus_{h \in \hat{S}} \UU_h: \sV^{(W,G)} \to \sV^{\cW}.
\]

Expressing $\UU$ in terms of the bases 
\[
\{\phi_h\}_{h \in \hat{S}} \text{ and }\{\ii_g, \ii_g {H}, \ldots , \ii_g {H}^{\dim(\cW_g)}\}_{g \in \tilde{S}},
\] 
$\UU$ takes the form of a block matrix
which is zero away from the diagonal blocks.  The first diagonal block (corresponding to the non--twisted sector of $\cW$) is size $4 \times 4$ and all others are $2 \times 2$.  Each block is nonsingular, thus $\UU$ is also.

Furthermore, one can check via a direct calculation on blocks that $\UU$ is symplectic.  This proves the following.
\begin{theorem}\label{t:U}
There is a $\CC[z, z^{-1}]$--valued degree--preserving symplectic transformation $\UU$ identifying $\sV^{(W,G)}$ with $\sV^{\cW}$.  Furthermore, for $h \in \hat{S}$ satisfying $\deg(\phi_h) \leq 2$, 
\[
\UU\left(c_h \cdot I^{(W,G)}_h(t, z)\right) = I^{\cW '}_{\mu^{-1}(h)}(t,z)
\]
where $c_h$ is the factor $\frac{t}{5}$ or $\frac{t^2}{25}$ depending on $h$.
\end{theorem}

\subsubsection{The main theorem}

By equation~\eqref{e:Jgens}, the slice $\sL^{\cW} \cap L_{J^{\cW}(s,-z)}$ of the ruling is generated  by $z\frac{\partial}{\partial s^i} J^{\cW}(s,z)$ where $i$ runs over a basis of $H^{even}_{CR}(\cW)$.  Thus the small slice  $\sL^{\cW}_{small}$ of the Lagrangian cone is completely determined by the first derivatives of $J^{\cW}(\bs, z)$ evaluated at points $sH \in H^2(\cW)$.  This implies the following:

\begin{lemma}
The small slice of the Lagrangian cone $\sL^{\cW}$ is determined by \[ \{I^{\cW}_g(q, z)\}_{\{g \in \tilde{G} | \deg{\ii_g} \leq 2\} }.\]
\end{lemma}

\begin{proof}
For $s \in H^2(\cW)$, each point of $\sL^{\cW} \cap L_{J^{\cW}(s,-z)}$ is of the form 
\[ z\sum_{i \in I} c_i(s, z)  \frac{\partial}{\partial s^i} J^{\cW}(s, -z).\]
where $I$ is a choice of basis for $H^{even}_{CR}(\cW)$.  By choosing a particular basis, we will show that such linear combinations are completely determined by the set $\{I^{\cW}_g(q, z)\}_{\{g \in \tilde{G} | \deg{\ii_g} \leq 2\} }$.

The main result of \cite{LSh} states that after choosing suitable coordinates (i.e. the mirror transformation) the $I$--functions $I^{\cW}_g$ and their derivatives give the rows of the solution matrix of $\nabla^{\cW}_s$ for $\cW$ when restricted to $ H^2(\cW)$.  Here $\nabla^{\cW}_s$ denotes the Dubrovin connection, defined in terms of the quantum cohomology of $\cW$ (see \cite{CK} and \cite{LSh}).  We summarize the content of the theorem below. Consider the subset of 
$H^{even}_{CR}(\cW)$ given by 
\[\{(\nabla^{\cW}_s)^k 1\}_{0 \leq k \leq 3} \cup \{\ii_g, \nabla^{\cW}_s \ii_g\}_{\deg(\ii_g) = 2}.
\]
We can check that this set forms a basis by using properties of the $J$--function.  Note first that the elements of $\{1\} \cup \{\ii_g\}_{\deg(\ii_g) = 2}$ are linearly independent.  For $ s \in H^2_{un}(\cW)$, $\nabla^{\cW}_s \ii_g = \frac{1}{z} H*_s \ii_g$ is a degree four class supported in a particular component of $I\cW$.  If $g = (0,0,0,r_1, r_2)$, $\frac{1}{z} H*_s \ii_g$ is a multiple of $\ii_gH$, and if $g = (0,0,r_1, r_1, r_2)$, $\frac{1}{z} H*_s \ii_g$ is a multiple of $\ii_{g_1}$ where $g_1 = (\langle -r_1\rangle, \langle -r_1\rangle, 0,0,\langle r_2 - r_1\rangle)$.  We can check that these multiples are non--zero by observing that the periods of $\frac{1}{z} H*_s \ii_g$ are obtained as the coefficients of $\frac{d}{ds} J^{\cW}_g(s, z)$ (see Definition~1.4 in \cite{LSh}) which are nonzero by \eqref{e:A-model}.  This shows that $\{\ii_g, \nabla^{\cW}_s \ii_g\}_{\deg(\ii_g) = 2}$ are linearly independent.  Similarly, $(\nabla^z_s)^k 1$ is a nonzero class of degree $k$ supported on the non twisted sector.  We conclude that 
\[
\{(\nabla^{\cW}_s)^k 1\}_{0 \leq k \leq 3} \cup \{\ii_g, \nabla^{\cW}_s \ii_g\}_{\deg(\ii_g) = 2}
\] 
is a set of $204 = \dim(H^{even}_{CR}(\cW))$ linearly independent elements and thus forms a basis.

By definition, for  $\ii_g$ of degree at most $2$,
\[
z\frac{\partial}{\partial s^g} J^{\cW}(\bs, -z)|_s= J^{\cW}_g(s, -z).
\]  
Because the $J$--function satisfies the quantum differential equation (equation~5 in \cite{G3}), if ${s^g}'$ is the dual coordinate to $\nabla^{\cW}_s \ii_g$, we have the following
\[
z\frac{\partial}{\partial {s^g}'} J^{\cW}(\bs, -z)|_s= \frac{d}{ds} J^{\cW}_g(s, -z).
\]
   
Similarly if $s^k$ is dual to $(\nabla^{\cW}_s)^k 1$, 
\[
z\frac{\partial}{\partial s^k} J^{\cW}(\bs, -z)|_s= \left(\frac{d}{ds}\right)^k J^{\cW}_e(s, -z).
\]
Therefore, for $s \in H^2(\cW)$, $\sL^{\cW} \cap L_{J^{\cW}(s,-z)}$ is completely determined by the $\CC[z]$--span of $\{J^{\cW}_g(s, -z)\}_{\age(g) \leq 2}$.

But, by the mirror theorem \eqref{e:A-model}, the span of $J^{\cW}_g(s, -z)$ is equal to the span of $I^{\cW}_g(\sigma^{-1}(s), -z)$ where $\sigma$ is the mirror map. 
\end{proof}

In FJRW theory, we have the analogous result.  

\begin{lemma}
The small slice of the Lagrangian cone $\sL^{(W,G)}$ is determined by 
\[
\{I^{(W,G)}_h(t, z)\}_{\{h \in \hat{S}| \deg(\phi_h) \leq 2, h \neq \jw^2\}}.
\]
\end{lemma}

\begin{proof} The proof is essentially the same as in the previous lemma.
\end{proof}

\begin{theorem}\label{t:main}
The symplectic transformation $\UU$ identifies the analytic continuation of the small slice of $\sL^{\cW}$ with the small slice of $\sL^{(W,G)}$.
\end{theorem}

\begin{proof}
The result follows immediately from Theorem~\ref{t:U} and the previous two lemmas.
\end{proof}

\begin{remark}
 Theorem~\ref{t:main} proves the first part of Conjecture~\ref{conj} restricted to the small parameters $sH \in H^{even}_{CR}(\cW)$ and $t\phi_{\jw^2} \in \sH^{nar}_{W,G}$ (see \ref{ss:slice}).  Note that although we have restricted all calculations to the small parameters, this is enough to completely determine $\UU$.
\end{remark}